\UseRawInputEncoding% AMS-LaTeX 1.2
\documentclass[11pt]{amsart}
\usepackage{nccmath}

\usepackage{amsmath,amsthm,amssymb}
\usepackage{times}
\usepackage{enumerate}
\usepackage{color}
\newcommand{\ubar}{\underline}

\newtheorem{prop}{Proposition}
\newtheorem{theo}[prop]{Theorem}
\newtheorem{lemm}[prop]{Lemma}

\newtheorem{rmk}[prop]{Remark}

\newcommand{\be}{\begin{equation}}
\newcommand{\ee}{\end{equation}}
\newcommand{\lt}{\left}
\newcommand{\rt}{\right}
\newcommand{\goto}{\rightarrow}
\newcommand{\al}{\alpha}
\newcommand{\e}{\epsilon}
\renewcommand{\leq}{\leqslant}
\renewcommand{\geq}{\geqslant}
\newcommand{\td}{\tilde}
\newcommand{\R}{\mathbb{R}}
\newcommand{\M}{\mathcal{M}}
\newcommand{\ka}{\kappa}
\newcommand{\z}{\mathbf{z}}

\newcommand{\ga}{\gamma}
\newcommand{\p}{\partial}

\newcommand{\lus}{\ubar{u}^*}
\newcommand{\lu}{\ubar{u}}
\newcommand{\uu}{\bar{u}}
\newcommand{\gas}{\gamma^*}

\newcommand{\ju}[2]{\begin{array}{#1}#2\end{array}}

\numberwithin{equation}{section}

\begin{document}
\setlength{\baselineskip}{1.2\baselineskip}

\title[Entire convex curvature flow]
{Entire convex curvature flow in Minkowski space}

\author{Zhizhang Wang}
\address{School of Mathematical Science, Fudan University, Shanghai, China}
\email{zzwang@fudan.edu.cn}
\author{Ling Xiao}
\address{Department of Mathematics, University of Connecticut,
Storrs, Connecticut 06269}
\email{ling.2.xiao@uconn.edu}
\thanks{Research of the first author is  sponsored by Natural Science  Foundation of Shanghai, No.20JC1412400,  20ZR1406600 and supported by NSFC Grants No.11871161.}

\begin{abstract}
In this paper, we study fully nonlinear curvature flows of noncompact spacelike hypersurfaces in Minkowski space.
We prove that if the initial hypersurface satisfies certain conditions, then the flow exists for all time. Moreover, we show that after rescaling
the flow converges to the future timelike hyperboloid, which is a self-expander.

\end{abstract}

\maketitle

\section{Introduction}
\label{int}
Let $\R^{n, 1}$ be the Minkowski space with the Lorentzian metric
\[ds^2=\sum\limits_{i=1}^ndx_i^2-dx_{n+1}^2.\]
In this paper, we study fully nonlinear curvature flows of noncompact spacelike hypersurfaces in Minkowski space.
Spacelike hypersurfaces $\M\subset\R^{n, 1}$ have an everywhere timelike normal
field, which we assume to be future directed and to satisfy the condition
$\lt<\nu, \nu\rt> =-1.$ Such hypersurfaces can be locally expressed as the graph of a function
$u: \R^n\rightarrow\R$ satisfying $|Du(x)|< 1$ for all $x\in\R^n.$

We consider a family of spacelike embeddings
\[X(\cdot, t):\R^n\rightarrow\R^{n, 1}\]
with corresponding hypersurfaces $\M(t)=\{(x, u(x, t))\mid x\in\R^n\}$ satisfying the evolution equation
\be\label{int.1}
\frac{\p X}{\p t}=F^\al\nu.
\ee
Here $\al>0,$ $F=f(\ka_1, \cdots, \ka_n)=\frac{1}{f_*(\ka_1^{-1},\cdots,\ka_n^{-1})},$ and $f_*$ satisfies the following conditions:
\be\label{condition 2}
\mbox{$f_*$ is a concave function in $\Gamma_n^+,$}
\ee
where $\Gamma_n^+:=\{\ka\in\R^n: \mbox{each component $\ka_i>0$}\},$
\be\label{condition 3}
f_*>0\,\,\mbox{in $\Gamma_n^+,$ $f_*=0$ on $\p\Gamma_n^+,$}
\ee
and
\be\label{condition 1}
f_*^i:=\frac{\p f_*}{\p\ka_i}>0\,\,\mbox{ in $\Gamma_n^+,$ $1\leq i\leq n.$}
\ee
In addition,
we shall assume $f_*$ satisfies the following more technical assumptions. These include
\be\label{condition 4}
\mbox{$f_*$ is homogeneous of degree one,}
\ee
\be\label{condition 5}
\mbox{there exists $\epsilon_0>0$ such that}
\sum f_*^i\ka_i^2\geq \epsilon_0f_*\sum\ka_i,
\ee
and for every $C>0$ and every compact set $E$ in $\Gamma_n^+,$ there exists $R=R(E, C)>0$ such that
\be\label{condition 6}
f_*(\ka_1, \cdots, \ka_{n-1}, \ka_n+R)>C, \forall \ka\in E.
\ee
Finally, for our convenience, we also suppose $f_*$ is normalized
\be\label{condition 7}
f_*(1,\cdots, 1)=1.
\ee
All of these assumptions are satisfied by $f_*=(s_n)^{\frac{\beta}{n}}(s_k)^{\frac{1-\beta}{k}},$ where $0\leq k\leq n,$ $\beta\in(0, 1],$
and
\[s_k=\frac{1}{{n\choose k}}\sum\limits_{1\leq i_1<\cdots<i_k\leq n}\ka_{i_1}\cdots\ka_{i_k}\]
is the normalized $k$-th elementary symmetric polynomial ($s_0=1$). One can find more examples of $f_*$ in \cite{GS04}.

We also assume the initial hypersurface $\M_0=\{(x, u_0(x))\mid x\in\R^n\}$ satisfyies:\\
\textbf{Conditions A:}
(1) strictly convex,
(2) spacelike,
(3) the curvatures of $\M_0,$ denoted by  $\ka[\M_0]=(\ka_1, \cdots, \ka_n),$ and their derivatives are uniformly bounded.
We also need the following condition:
\[\tag{4} \sqrt{|x|^2+C_0}<u_0(x)<\sqrt{|x|^2+C_1}\label{initial-condition},\]
for some $C_1>C_0>0.$

Note that condition \eqref{initial-condition} is needed to obtain the local $C^1$ estimates (see Subsection \ref{lc1}).
\begin{rmk}
\label{rmk1}
It is clear that condition \eqref{initial-condition} implies
$$u_0(x)\rightarrow |x|, \text{ as } |x|\rightarrow +\infty.$$ In view of Lemma 14 of \cite{WX20}, we obtain the Legendre transform of $u_0$ which we denote by
$u_0^*,$ satisfies $u_0^*=0$ on $\p B_1.$ By the strict convexity of $u_0^*$ we have $u_0^*<0$ in $B_1.$
\end{rmk}

There is a great deal of literature concerning the entire graphical curvature flow. In Minkowski space, Ecker \cite{Ecker} proved the long time existence for entire graphical solutions of the mean curvature flow. In \cite{Aa}, Aarons studied the forced mean curvature flow of entire graphs in Minkowski space. With initial data $\M_0$ being a smooth strictly spacelike hypersurface, Aarons proved long time existence of solutions to the flow where the forcing term is smooth. Furthermore,  Aarons also obtained some convergence results when the forcing term is a constant and $\M_0$ has bounded curvatures. In \cite{BS09}, Bayard and Schn\"{u}rer constructed solutions to the logarithmic Gauss curvature flow for convex spacelike hypersurfaces with prescribed Gauss map images. They showed that the solutions converge to the constant Gauss curvature hypersurfaces. In Euclidean space, entire graphical solutions of the mean curvature flow have been studied by Ecker and Huisken in \cite{EH}. Entire graphical solutions of fully-nonlinear flows were recently studied by Daskalopoulos and her collaborators in
\cite{CDKL, DH}.

The flow \eqref{int.1} is well understood when the initial entire graphical hypersurface is co-compact. It was shown in \cite{ACFMc} that for any co-compact, spacelike, uniformly locally convex initial embedding $\M_0$, there exists a unique co-compact solution $\M_t$ satisfying equation \eqref{int.1} for $t\in(0, \infty).$ Moreover, after rescaling, $\M_t$ converges in $C^\infty$ to an embedding with image equal to the future timelike hyperboloid $\mathbb{H}^n.$
As it was proved in Section 12 of \cite{ACFMc}, this result implies a convergence result for a class of special solution of the cross-curvature flow (see \cite{CH}) of negatively curved Riemannian metric on three-manifolds.

This work is concerned with the long time existence and convergence of \eqref{int.1} for initial data
$\M_0$ which is an entire graph satisfying \textbf{Conditions A}. In particular, we prove the following theorem.

\begin{theo}
\label{int-them1}
Let $\M_0$ be an entire spacelike hypersurface satisfying \textbf{Conditions A,} and the dual of function $f,$ i.e., $f_*,$ satisfies \eqref{condition 1}-\eqref{condition 6}.  Then, given $\M_0,$ there exists a unique $C^\infty$
smooth, entire, spacelike, strictly convex solution $\M_t:=\{(x, u(x, t))| (x, t)\in\R^n\times(0, \infty)\}$ of the equation $\eqref{int.1}$ satisfying for any fixed $t\in(0, \infty),$ $u(x, t)-|x|\goto 0$ as $|x|\goto \infty.$
The rescaled embeddings given by $\td{X}=\frac{X}{[(1+\al)F(I)\td{t}]^{1/(1+\al)}}$ converge in $C^\infty$ to an embedding with image equal
to the future timelike hyperboloid $\mathbb{H}^n,$ where $\td{t}=(1+\al)^{-1}+t$ and $I$ is the identity matrix.
\end{theo}

\begin{rmk}
\label{intrmk0}
Notice that here we study the complete flow instead of the co-compact flow, we need to deal with the behavior of the flow at infinity. Therefore, our techniques here are completely different from those in \cite{ACFMc}. Moreover, in order to obtain that this flow converges, we need to develop new methods to obtain interior estimates. The methods we developed here are completely different from those used in studying curvature flows with forcing terms (see \cite{Aa, BS09} for examples).
\end{rmk}

\begin{rmk}
\label{intrmk1}
In Section \ref{conv} we will see that the rescaled embeddings $\td{X}$ converge to a self-expander satisfying
\be\label{rmk.1}F^\al=-\lt<X, \nu\rt>,\ee which happens to be the future timelike hyperboloid. On the other hand, for any hypersurface $\M$ satisfies
 \eqref{rmk.1}, we have $[(1+\alpha)\td{t}]^{\frac{1}{1+\alpha}}X$ satisfies flow equation \eqref{int.1}. By the result in \cite{WX22}, we know that there exist infinitely many hypersurfaces that satisfying \eqref{rmk.1}. Therefore, in order for the rescaled flow to converge to the hyperboloid, the condition \eqref{initial-condition} in \textbf{Conditions A} is necessary.
\end{rmk}

The organization of the paper is as follows. In Section \ref{special solution} we give a special solution to equation \eqref{int.1}. This solution inspires us to
formulate the approximate problem and construct barrier functions in later sections. In Section \ref{approximate problem} we prove the long time existence of the approximate problem. Local $C^1$ and $C^2$ estimates are established in Section \ref{localest}. Combining with Section \ref{approximate problem}, we prove the existence of solution to \eqref{int.1} for all time. In Section \ref{conv} we show that after rescaling the solution of \eqref{int.1} converges to the unit future hyperboloid as $t\goto\infty.$

\bigskip

\section{Special solution of equation \eqref{int.1}}
\label{special solution}
In this section, we will introduce a special solution for equation \eqref{int.1}. We will also discuss the Legendre Transform of this special solution.
In later sections, we will use this special solution to construct various barrier functions.

Let's first note that since $\M$ is spacelike, the position vector of $\M$ can be written as $X=(x, u(x)).$
After reparametrization, we can rewrite \eqref{int.1} as following equation
\be\label{ss.1}
\left\{
\begin{aligned}
\frac{\p u}{\p t}&=F^\al\lt(\frac{1}{w}\gamma^{ik}u_{kl}\gamma^{lj}\rt)w\\
u(\cdot, 0)&=u_0(x),
\end{aligned}
\right.
\ee
where $w=\sqrt{1-|Du|^2},$ $\gamma^{ik}=\delta_{ik}+\frac{u_iu_k}{w(1+w)},$ and $u_{kl}=D^2_{x_kx_l}u$ is the ordinary Hessian of $u.$
Let $\z=\sqrt{|x|^2+[(1+\al) t]^{\frac{2}{1+\al}}},$ then a straightforward calculation yields
\[\z_t=\lt(|x|^2+[(1+\al)t]^{\frac{2}{1+\al}}\rt)^{-\frac{1}{2}}[(1+\al)t]^{\frac{1-\al}{1+\al}},\]
\[\kappa[\M_{\z}]=[(1+\al)t]^{-\frac{1}{1+\al}}(1, 1, \cdots, 1),\]
\[F^\al(\ka[\M_{\z}])=[(1+\al)t]^{-\frac{\al}{1+\al}},\]
and
\[w=\sqrt{1-\frac{|x|^2}{|x|^2+[(1+\alpha) t]^{\frac{2}{1+\al}}}}=\frac{[(1+\al)t]^{\frac{1}{1+\al}}}{\z}.\]
Therefore, $\z=\sqrt{|x|^2+[(1+\al)t]^{\frac{2}{1+\al}}}$ satisfies \eqref{ss.1} for $t\in (0, \infty).$ We can see that, for $\M_t=\{(x, \z(x, t))\mid (x, t)\in\R^n\times\R_+\},$ $\ka[\M_t]\goto 0$ as $t\goto\infty.$ Moreover, let
$\tilde{X}=\frac{X}{[(1+\al)t]^{\frac{1}{1+\al}}}$ where $X$ is the position vector of $\M_t.$ Then, for any fixed $t>0,$ the rescaled graph $\td{\M}(t)=\{(x, \sqrt{|x|^2+1})\mid x\in\R^n\}$ is the standard hyperboloid satisfying $F^\al(\ka[\M])=-\lt<X, \nu\rt>.$

Next, let's consider the Legendre transform of $\z.$  By definition we have
\begin{align*}
\z^*&=x\cdot D\z-\z=-\frac{[(1+\alpha) t]^{\frac{2}{1+\al}}}{\sqrt{|x|^2+[(1+\alpha) t]^{\frac{2}{1+\al}}}}\\
&=-[(1+\al)t]^{\frac{1}{1+\al}}\sqrt{1-|D\z|^2}=-[(1+\al)t]^{\frac{1}{1+\al}}\sqrt{1-|\xi|^2}.
\end{align*}
It's easy to verify that $\z^*$ is a solution of
\be\label{ss.2}
\lt\{
\begin{aligned}
\z^*_t&=-F_*^{-\al}(w^*\gamma^*_{ik}\z^*_{kl}\gamma^*_{lj})w^*\,\,&\mbox{in $B_1\times(0, \infty)$}\\
\z^*(\cdot, t)&=0\,\,&\mbox{on $\partial B_1\times[0, \infty)$}\\
\z^*(\cdot, 0)&=0 \,\,&\mbox{in $ B_1\times\{0\}$},
\end{aligned}
\rt.
\ee
where $w^*=\sqrt{1-|\xi|^2},$ $\gas_{ik}=\delta_{ik}-\frac{\xi_i\xi_k}{1+w^*},$ $u^*_{kl}=\frac{\p^2u^*}{\p\xi_k\p\xi_l},$
and \[F_*^{-\al}(w^*\gas_{ik}u^*_{kl}\gas_{lj})=f_*^{-\al}(\ka^*[w^*\gas_{ik}u^*_{kl}\gas_{lj}]).\] Here
$\ka^*[w^*\gas_{ik}u^*_{kl}\gas_{lj}]=(\ka_1^*, \cdots, \ka_n^*)$ are the eigenvalues of the matrix $(w^*\gas_{ik}u^*_{kl}\gas_{lj}).$
Since when $\M_t$ satisfies equation $\eqref{int.1},$ its support function $v=\lt<X, \nu\rt>$ satisfies
 $v_t=-F^{\al}(\ka[\M_t]).$ In this paper we will consider the solvability of the following problem:
\be\label{ss.3}
\lt\{
\begin{aligned}
u^*_t&=-F_*^{-\al}(w^*\gamma^*_{ik}u^*_{kl}\gamma^*_{lj})w^*\,\,&\mbox{in $B_1\times(0, \infty)$}\\
u^*(\cdot, t)&=0\,\,&\mbox{on $\partial B_1\times[0, \infty)$}\\
u^*(\cdot, 0)&=u_0^*\,\,&\mbox{on $B_1\times\{0\}.$}
\end{aligned}
\rt.
\ee

\bigskip
\section{Solvability of the approximate problem }
\label{approximate problem}
Notice that equation \eqref{ss.3} is degenerate. Inspired by the special solution $\mathbf{z}^*$ of \eqref{ss.2}, in this section we will consider the solvability of the following approximate problem:
\be\label{ap.1}
\lt\{
\begin{aligned}
(u^*_r)_t&=-F_*^{-\al}(w^*\gamma^*_{ik}u^*_{kl}\gamma^*_{lj})w^*\,\,&\mbox{in $B_r\times(0, T]$}\\
u^*_r(\cdot, t)&=\hat{u}^*_0-[(1+\al)\td{t}]^{\frac{1}{1+\al}}\sqrt{1-r^2}\,\,&\mbox{on $\p B_r\times[0, T]$}\\
u^*_r(\cdot, 0)&=\hat{u}^*_0-\sqrt{1-|\xi|^2}=u_0^*\,\,&\mbox{on $B_r\times\{0\},$}
\end{aligned}
\rt.
\ee
where $\td{t}=t+(1+\al)^{-1}$ and $T>0$ is an arbitrary positive number.
\begin{rmk}
\label{rmk2}
Note that for our convenience, in the rest of this paper, we always assume $\hat{u}_0^*=u_0^*+\sqrt{1-|\xi|^2}$ to be strictly convex. For if it's not,
we can consider the following equation instead
\be\label{ap.1'}
\lt\{
\begin{aligned}
(u^*_r)_t&=-F_*^{-\al}(w^*\gamma^*_{ik}u^*_{kl}\gamma^*_{lj})w^*\,\,&\mbox{in $B_r\times(0, T]$}\\
u^*_r(\cdot, t)&=\hat{u}^*_0-\e[(1+\al)\td{t}]^{\frac{1}{1+\al}}\sqrt{1-r^2}\,\,&\mbox{on $\p B_r\times[0, T]$}\\
u^*_r(\cdot, 0)&=\hat{u}^*_0-\e\sqrt{1-|\xi|^2}=u_0^*\,\,&\mbox{on $B_r\times\{0\},$}
\end{aligned}
\rt.
\ee
where $\e>0$ is a small constant such that $\hat{u}_0^*=u_0^*+\e\sqrt{1-|\xi|^2}$ is strictly convex.
\end{rmk}

\subsection{$C^0$ estimates for $u_r^*$-- solution of \eqref{ap.1}}
\label{c0forur*}
In this subsection, we will establish the uniform upper and lower bounds
for the solution $u_r^*$ of \eqref{ap.1}.

Let's denote $\lu^s=\sqrt{|x|^2+a_0^2}$ and $\uu^s=\sqrt{|x|^2+a_1^2},$ where
$a_0^2< \min \{C_0, 1\},$ $a_1^2> \max\{C_1, 1\},$ and the constants $C_0,$ $C_1$ are the same constants as in \eqref{initial-condition} of \textbf{Conditions A}. We will denote the Legendre transform of $\uu^s,$ $\lu^s$ by $\uu^{s*},$ $\lu^{s*}$ respectively. Applying Lemma 13 of \cite{WX20} we know $\uu^{s*}<u_0^*<\lu^{s*}.$
Moreover, a straightforward calculation yields
\[F^{\al}(\ka[\M_{\uu^s}])=a_1^{-\al}<-\lt<X_{\uu^s}, \nu_{\uu^s}\rt>=a_1,\,\,\mbox{for $a_1>1$},\]
where $X_{\uu^s},\nu_{\uu^s}$ are position vector and normal vector of $\M_{\uu^s}$.
We get  $\uu^{s*}$ satisfies
\[-F^{-\al}_*(w^*\gamma^*_{ik}\uu^{s*}_{kl}\gamma^*_{lj})\sqrt{1-|\xi|^2}>\uu^{s*}.\]
Consider $\td{\uu}^*=[(1+\al)\td{t}]^{\frac{1}{1+\al}}\uu^{s*},$ we have
\be\label{ap.2}
\begin{aligned}
\td{\uu}^*_t&=[(1+\al)\td{t}]^{-\frac{\al}{1+\al}}\uu^{s*}\\
&<-[(1+\al)\td{t}]^{-\frac{\al}{1+\al}}F_*^{-\al}(w^*\gamma^*_{ik}\uu^{s*}_{kl}\gamma^*_{lj})w^*\\
&=-F_*^{-\al}(w^*\gamma^*_{ik}\td{\uu}^*_{kl}\gamma^*_{lj})w^*.
\end{aligned}
\ee
\begin{lemm}
\label{c0fur-lem1*}
Denote $u_b^{r*}=\td{\uu}^*-[(1+\al)\td{t}]^{\frac{1}{1+\al}}\sqrt{1-|\xi|^2},$ then $u_b^{r^*}$ is a subsolution of \eqref{ap.1}
on $\bar{B}_r\times[0, T]$ for any $T>0.$
\end{lemm}
\begin{proof}
On $B_r\times\{0\}$ we have
\[u_b^{r*}(\cdot, 0)=\uu^{s*}-\sqrt{1-|\xi|^2}<u_0^*\,\,\mbox{in $B_r$,}\]
where $\uu^{s^*}=-a_1\sqrt{1-|\xi|^2}.$

On $\partial B_r\times[0, \infty),$
\[u_b^{r*}(\cdot, t)=[(1+\al)\td{t}]^{\frac{1}{1+\al}}\uu^{s*}
-[(1+\al)\td{t}]^{\frac{1}{1+\al}}\sqrt{1-r^2}<\hat{u}_0^*-[(1+\al)\td{t}]^{\frac{1}{1+\al}}\sqrt{1-r^2}.\]
Moreover, by earlier discussion it's not hard to see that
\[\frac{\partial u_b^{r*}}{\partial t}+F_*^{-\al}(w^*\gamma^*_{ik}(u_b^{r*})_{kl}\gamma^*_{lj})w^*<0.\]
Therefore, Lemma \ref{c0fur-lem1*} is proved.
\end{proof}
Similarly , we can show
\begin{lemm}
\label{c0fur-lem1}
 Denote $u_s^{r^*}=[(1+\alpha)\td{t}]^{\frac{1}{1+\al}}\lu^{s*},$ then $u_s^{r*}$ is a supersolution of \eqref{ap.1}
 on $\bar{B}_r\times[0, T]$ for any $T>0.$
\end{lemm}
\begin{proof}
On $B_r\times\{0\}$ we have
\[u^{r*}_s(\cdot, 0)=\lu^{s*}>u_0^*.\]
On $\partial B_r\times(0, T]$ we want to show
\[[(1+\al)\td{t}]^{\frac{1}{1+\al}}\lu^{s*}>\hat{u}_0^*-[(1+\al)\td{t}]^{\frac{1}{1+\al}}\sqrt{1-r^2}.\]
This is equivalent to show
\[[(1+\al)\td{t}]^{\frac{1}{1+\al}}\lt(-a_0\sqrt{1-r^2}+\sqrt{1-r^2}\rt)>\hat{u}_0^*\]
we can always choose $0<a_0<1$ small such that the $\textbf{l.h.s}>0.$ By our assumption on the initial hypersurface, we have $\hat{u}_0^*$ is a strictly convex function with $\hat{u}_0^*=0$ on $\p B_1,$ this implies
$\textbf{r.h.s}<0.$ We have shown $u^{r*}_s>u_r^*$ on the parabolic boundary. Moreover, it's easy to see that
$u^{r*}_s$ satisfies
\[\lt(u^{r*}_s\rt)_t>-F_*^{-\al}\lt(w^*\gas_{ik}(u^{r*}_s)_{kl}\gas_{lj}\rt)w^*.\] This completes the proof of the Lemma.
\end{proof}

Combining Lemma \ref{c0fur-lem1*} and \ref{c0fur-lem1}, when we apply the maximum principle we obtain that
\begin{lemm}
\label{c0fur*-lem}
Let $u_r^*$ be the solution of equation \eqref{ap.1} on $\bar{B}_r\times [0, T]$ then
\[u_b^{r*}<u_r^*<u_s^{r*}.\]
\end{lemm}

\subsection{$C^1$ estimates for $u_r^*$-- solution of \eqref{ap.1}}
\label{c1forur*}
In this subsection, we will apply Anmin Li's idea \cite{AML} to construct barrier functions and establish the gradient estimates
for the solution $u_r^*$ of \eqref{ap.1}.
\begin{lemm}
\label{c1forur*-lem1}
Let $\varphi(\xi)$ be a strictly convex, smooth function defined on $B_1,$ we will show for any $\hat{\xi}\in\partial B_r,$ $r\in(0, 1),$
there exists $\ubar{\mathfrak{v}}=b_1\xi_1+\cdots+b_n\xi_n+d$ such that

(1). $\ubar{\mathfrak{v}}(\hat{\xi})=\varphi(\hat{\xi}),$ and

(2). $\ubar{\mathfrak{v}}(\xi)<\varphi(\xi)$ for any $\xi\in\bar{B}_r\setminus\{\hat{\xi}\}.$\\
Here $b_i=b_i(|\varphi|_{C^1})\,\,\mbox{for $1\leq i\leq n,$}$ and $d=d(|\varphi|_{C^1}).$
\end{lemm}
\begin{proof} Without loss of generality we assume $\hat{\xi}=(r, 0, \cdots, 0).$ Since
$\varphi$ is strictly convex we obtain
\[\varphi(\xi)>\varphi(r, 0, \cdots, 0)+\varphi_1(r, 0, \cdots, 0)(\xi_1-r)+
\sum\limits_{i=2}^n\varphi_i(r, 0, \cdots, 0)\xi_n\,\,\mbox{for any $\xi\in\bar{B}_r\setminus\{\hat{\xi}\}.$} \]
Let $\ubar{\mathfrak{v}}=\varphi(r, 0, \cdots, 0)+\varphi_1(r, 0, \cdots, 0)(\xi_1-r)+
\sum\limits_{i=2}^n\varphi_i(r, 0, \cdots, 0)\xi_n,$ then we are done. We also note that $b_i$ depends on $|\varphi|_{C^1}$ for $1\leq i\leq n$
and $d$ depends on $|\varphi|_{C^1}.$
\end{proof}

\begin{lemm}
\label{c1forur*-lem2}
Let $\varphi(\xi)$ be a strictly convex, smooth function defined on $B_1,$ we will show for any $\hat{\xi}\in\partial B_r,$ $r\in(0, 1),$
there exists $\bar{\mathfrak{v}}=b_1\xi_1+\cdots+b_n\xi_n+d$ such that

(1). $\bar{\mathfrak{v}}(\hat{\xi})=\varphi(\hat{\xi}),$ and

(2). $\bar{\mathfrak{v}}(\xi)>\varphi(\xi)$ for any $\xi\in\bar{B}_r\setminus\{\hat{\xi}\}.$\\
Here $b_1$ and $d$ depend on $|\varphi|_{C^2},$ while $b_i=b_i(|\varphi|_{C^1})$ for $2\leq i\leq n.$
\end{lemm}
\begin{proof}
Following the idea of \cite{AML}, we consider an arbitrary great circle $c(t)$ on $\partial B_r$ passing through $\hat{\xi}.$ Without loss of generality, we assume
$\hat{\xi}=(r, 0, \cdots, 0)$ and the great circle is
\[\xi_1=r\cos s,\,\, \xi_2=r\sin s,\,\, \xi_3=\cdots=\xi_n=0,\,\,-\pi\leq s\leq\pi.\]
Consider $\hat{F}(s)=b_1r\cos s+b_2r\sin s+d-\varphi(r\cos s, r\sin s, 0, \cdots, 0),$ then
\[\hat{F}(0)=b_1r+d-\varphi(r, 0, \cdots, 0).\]
We let $b_1=\frac{\varphi(r, 0, \cdots, 0)-d}{r}$ so that $\hat{F}(0)=0.$ Now differentiating $\hat{F}$ with respect to $s$ we get
\[\frac{d\hat{F}}{ds}=-b_1r\sin s+b_2r\cos s+\varphi_1r\sin s-\varphi_2r\cos s.\] Let $b_2=\varphi_2(r, 0, \cdots, 0)$ then we obtain
$\frac{d}{ds}\hat{F}(0)=0$ and
\be\label{c1f.1}
\frac{d^2\hat{F}}{ds^2}=(d-\varphi(r, 0, \cdots, 0))\cos s-b_2r\sin s+\frac{d^2\varphi}{ds^2}.
\ee
When $s\in[-\frac{\pi}{4}, \frac{\pi}{4}],$ i.e., $\cos s\geq\frac{\sqrt{2}}{2},$ choosing $d>0$  sufficiently large we have
\[\frac{d^2\hat{F}}{ds^2}\geq\frac{\sqrt{2}}{2}(d-\varphi(r, 0, \cdots, 0))-C(|\varphi|_{C^2})>0.\]
When $s\in[-\pi, -\frac{\pi}{4})\cup(\frac{\pi}{4}, \pi]$ and $d>0$ sufficiently large, we have
\be\label{c1f.2}
\begin{aligned}
\hat{F}(s)&=(-d+\varphi(r, 0, \cdots, 0))\cos s+b_2r\sin s+d-\varphi(r\cos s, r\sin s, 0, \cdots, 0)\\
&>d\lt(1-\frac{\sqrt{2}}{2}\rt)-C(|\varphi|_{C^1})>0.
\end{aligned}
\ee
Therefore, we have $\bar{\mathfrak{v}}(\xi)>\varphi(\xi)$ on $\partial B_r\setminus\{\hat{\xi}\}$ and
$\bar{\mathfrak{v}}(\hat{\xi})=\varphi(\hat{\xi}).$
In view of the strict convexity of $\varphi,$ it's easy to see that $\triangle(\bar{\mathfrak{v}}-\varphi)<0$ thus $\bar{\mathfrak{v}}(\xi)>\varphi(\xi)$ in $B_r$.
\end{proof}

\begin{lemm}
\label{c1forur*-lem3}
Let $u_r^*$ be the solution of equation \eqref{ap.1} on $\bar{B}_r\times [0, T],$ then
$|Du_r^*|\leq C(|u_0^*|_{C^2}, t)$ in $\bar{B}_r \times[0, T].$
\end{lemm}
\begin{proof}
We only need to show $|Du^*_r|$ is bounded on $\partial B_r.$ For any $\hat{\xi}\in \partial B_r,$ let
\[\lu^*_r=-[(1+\al)\td{t}]^{\frac{1}{1+\al}}\sqrt{1-|\xi|^2}+\ubar{\mathfrak{v}}\]
and \[\uu^*_r=-[(1+\al)\td{t}]^{\frac{1}{1+\al}}\sqrt{1-|\xi|^2}+\bar{\mathfrak{v}}.\]
Here $\ubar{\mathfrak{v}}$ and $\bar{\mathfrak{v}}$ are linear functions constructed in Lemma \ref{c1forur*-lem1} and Lemma \ref{c1forur*-lem2}
by letting $\varphi(\xi)=\hat{u}_0^*(\xi).$
Notice that $\lu^*_r,$ $\uu^*_r,$ and $u_r^*$ all satisfies the flow equation \eqref{ap.1}. Moreover, at the parabolic boundary
$\lt(B_r\times\{0\}\rt)\bigcup\lt(\partial B_r\times(0, T]\rt)$ we have
$\lu^*_r\leq u^*_r\leq\uu^*_r$ with equalities hold at $\hat{\xi}\times[0, T].$ By the maximum principle we obtain
\[\lu^*_r<u^*_r<\uu^*_r\,\,\mbox{in $B_r\times(0, T)$}.\]
Therefore,
\[|Du^*_r(\hat{\xi}, t)|<\max\lt\{|D\lu^*_r(\hat{\xi}, t)|, |D\uu_r^*(\hat{\xi}, t)|\rt\}.\]
this completes the proof of Lemma \ref{c1forur*-lem3}.
\end{proof}

Let $\mathbb{S}^{n-1}(r)=\{\xi\in\R^n\mid\sum\xi_i^2=r^2\},$ and we denote the angular derivative
$\xi_k\frac{\partial}{\partial\xi_l}-\xi_l\frac{\partial}{\partial\xi_k}$ on $\mathbb{S}^{n-1}(r)$
by $\partial_{k,l}$ or simply by $\partial$ when no confusion arises.
\begin{lemm}
\label{c1forur*-lem4} Let $u_0^*$ be the Legendre transform of the initial hypersurface $u_0,$ then
$|\partial u_0^*|$ is bounded by a constant depending on $|u_0^*|_{C^3}$.
\end{lemm}
\begin{proof}
Since $\det(D^2u_0^*)=\frac{1}{K(\xi)}(1-|\xi|^2)^{-\frac{n+2}{2}}$
we have
\[(u_0^*)^{ij}(u_0^*)_{kij}=\frac{\partial}{\partial\xi_k}\lt[\log(1-|\xi|^2)^{-\frac{n+2}{2}}\rt]-\frac{\partial\log K}{\partial\xi_k}.\]
By the assumption that $u_0^*$ satisfies \textbf{Conditions A} we get
\be\label{add1*}(u_0^*)^{ij}(\partial u_0^*)_{ij}=-\partial\log K\geq-nC_1\ee
for some positive $C_1$ depending only on
$|u_0^*|_{C^3}.$ Note that in the above calculation we used $\p |\xi|^2=0.$
The inequality \eqref{add1*} yields
\[(u_0^*)^{ij}(\partial u_0^*+C_1u_0^*)_{ij}\geq 0.\]
Thus, we obtain $\partial u_0^*\leq-C_1u_0^*.$
Similarly, we can show $\partial u_0^*\geq C_2u_0^*.$
\end{proof}

\begin{lemm}
\label{c1forur*-lem5} Let $u_0^*$ be the Legendre transform of the initial hypersurface $u_0,$ then
$\partial^2u_0^*$ is bounded from above by a constant depending on $|u_0^*|_{C^4}$.
\end{lemm}
\begin{proof}
Differentiating $(u_0^*)^{ij}(\partial u_0^*)_{ij}=-\partial\log K$ we get
\[(u_0^*)^{ij}\partial[(\partial u_0^*)_{ij}]+\partial(u^{*ij}_0)(\partial u_0^*)_{ij}=\partial^2\log K.\]
Following the proof of Lemma 5.2 of \cite{AML} we obtain,
\[(u_0^*)^{ij}[(\partial^2 u^*_0)_{ij}]\geq -nC_3,\]
for some positive $C_3$ depending only on
$|u_0^*|_{C^4}.$
This implies
\[\partial^2u_0^*\leq -C_3u_0^*.\]
\end{proof}

\subsection{Upper and lower bounds for $F_*$}In this subsection we will show that along the flow, $F_*$ is bounded from above and below.

We take the hyperplane $\mathbb{P}:=\{X=(x_1, \cdots, x_{n}, x_{n+1}) |\, x_{n+1}=1\}$ and consider the projection of
$\mathbb{H}^n(-1)$ from the origin into $\mathbb{P}.$ Then $\mathbb{H}^n(-1)$ is mapped in
a one-to-one fashion onto an open unit ball $B_1:=\{\xi\in\R^n |\, \sum\xi^2_k<1\}.$ The map
$P$ is given by
\[P: \mathbb{H}^n(-1)\rightarrow B_1;\,\,(x_1, \cdots, x_{n+1})\mapsto (\xi_1, \cdots, \xi_n),\]
where $x_{n+1}=\sqrt{1+x_1^2+\cdots+x_n^2},$ $\xi_i=\frac{x_i}{x_{n+1}}.$
Recall that when $u^{*}_r$ satisfies \eqref{ap.1}, let $v_r=\frac{u^{*}_r}{w^*}$ then $v_r$ satisfies
\be\label{lc1.3}
\left\{
\begin{aligned}
(v_r)_t&=-\td{F}^{-1}(\Lambda_{ij}):=\td{G}(\Lambda_{ij})\,\,&\mbox{in $P^{-1}(B_r)\times(0, T]:=U_r\times(0, T]$},\\
v_r&=\frac{\hat{u}_0^*}{\sqrt{1-r^2}}-[(1+\al)\td{t}]^{\frac{1}{1+\al}}\,\,&\mbox{on $\partial U_r\times[0, T]$},\\
v_r(\cdot, 0)&=\frac{\hat{u}_0^*}{\sqrt{1-|\xi|^2}}-1:=v_0\,\,&\mbox{on $U_r\times\{0\}.$}
\end{aligned}
\right.
\ee
Here, $\td{F}=F_*^{\al},$ $\Lambda_{ij}=\bar{\nabla}_{ij}v_r-v_r\delta_{ij},$ and $\bar{\nabla}_{ij}v$  denote covariant derivatives with respect to the hyperbolic metric. In the following, when there is no confusion, we will drop the subscription on $v_r$. Moreover, we also write $v_{ij}:=\bar{\nabla}_{ij}v.$

\begin{lemm}
\label{lc1-lem4-tG lower bound} Assume \eqref{lc1.3} is satisfied by $v$ on $\bar{U}_r\times[0, T],$ then
there exists $c_0>0$ depending only on the initial hypersurface $\M_0$ such that $\td{F}>\frac{1}{c_0}$ in $\bar{U}_r\times[0, T].$
\end{lemm}
\begin{proof}
From \eqref{lc1.3} and $\tilde{G}=-\tilde{F}^{-1}$, we derive
\be\label{lc1.4}
\begin{aligned}
\td{G}_t&=\td{G}^{ij}(\dot{v}_{ij}-\dot{v}\delta_{ij})\\
&=\td{G}^{ij}\bar{\nabla}_{ij}\td{G}-\td{G}\sum\td{G}^{ii}.
\end{aligned}
\ee
By the maximum principle we know that $\td{G}$ achieves its minimum at the parabolic boundary
$\lt(\p U_r\times\rt.$\\$\lt.(0, T]\rt)\cup\lt(\bar{U}_r\times\{0\}\rt).$
Since on $\partial U_r\times(0, T]$ we have
\[\td{F}=[(1+\al)\td{t}]^{\frac{\al}{1+\al}}\geq 1,\]
the Lemma is proved.
\end{proof}

\begin{lemm}
\label{lc1-lem7-upperbound of tF}
Assume \eqref{lc1.3} is satisfied by $v$ on $\bar{U}_r\times[0, T],$ then
there exists $c_1(t)>0$ depending only on the initial hypersurface $\M_0,$ $T,$ and time $t$  such that $\td{F}(\cdot, t)<\frac{1}{c_1(t)}$ in $\bar{U}_r\times[0, T].$
\end{lemm}
\begin{proof}
Let $\hat{G}=\td{F}^{-1}=-\td{G}$ and $\mathcal{L}:=\frac{\p}{\p t}-\td{G}^{ij}\bar{\nabla}_{ij}.$ Then we have
\[\mathcal{L}\hat{G}=\frac{\p}{\p t}\hat{G}-\td{G}^{ij}\bar{\nabla}_{ij}\hat{G}=-\hat{G}\sum \td{G}^{ii}.\]
Denote $\hat{v}=-v,$ and by our assumption on the initial data we know there exist $a_1>a_0>0$ such that
\[-a_1\sqrt{1-|\xi|^2}<u_0^*<-a_0\sqrt{1-|\xi|^2},\]
which yields $\max\hat{v}(\cdot, 0)<a_1.$
Moreover, by the Subsection \ref{c0forur*} we get
\[-a_1[(1+\al)\td{t}]^{\frac{1}{1+\al}}\sqrt{1-|\xi|^2}-[(1+\al)\td{t}]^{\frac{1}{1+\al}}\sqrt{1-|\xi|^2}
<u^*_r<-a_0[(1+\al)\td{t}]^{\frac{1}{1+\al}}\sqrt{1-|\xi|^2}.\]
This implies
\[a_0[(1+\al)\td{t}]^{\frac{1}{1+\al}}<\hat{v}(\cdot, t)<a_1[(1+\al)\td{t}]^{\frac{1}{1+\al}}+[(1+\al)\td{t}]^{\frac{1}{1+\al}}\]
and
\[a_0+[(1+\al)\td{t}]^{\frac{1}{1+\al}}-1<\hat{v}(\p U_r, t)<a_1+[(1+\al)\td{t}]^{\frac{1}{1+\al}}-1.\]
 In view of \eqref{lc1.3} we obtain
\[\mathcal{L}\hat{v}=(1+\al)\hat{G}-\hat{v}\sum \td{G}^{ii}.\]
Now consider the test function $\varphi=\frac{\hat{G}}{\hat{v}}.$
If $\varphi$ achieves its minimum at an interior point, then at this point we obtain
\be\label{lc1.11}
\begin{aligned}
\mathcal{L}\varphi&=\frac{\mathcal{L}\hat{G}}{\hat{v}}-\frac{\hat{G}}{\hat{v}^2}\mathcal{L}\hat{v}\\
&=-\varphi\sum\td{G}^{ii}-\frac{\hat{G}}{\hat{v}^2}\lt((1+\al)\hat{G}-\hat{v}\sum\td{G}^{ii}\rt)\\
&=-(1+\al)\varphi^2.
\end{aligned}
\ee
Without loss of generality let's assume $\varphi_{\min}\leq 1,$ then consider
$\td{\varphi}=e^{(1+\al)t}\varphi.$ We know that $\td{\varphi}$ doesn't achieve interior minimum that is less than or equal to $1.$
Let $c=\min F_*^{-\al}(\cdot, 0),$ then
\[e^{(1+\al)t}\varphi\geq\min\limits_{0\leq t\leq T}\lt\{\frac{c}{\max\hat{v}(\cdot, 0)}, \frac{[(1+\al)\td{t}]^{-\frac{\al}{1+\al}}}{\hat{v}(\p U_r, t)}, 1\rt\}.\]
Here, the three values on the right hand side are corresponding to when the minimum is achieved at $t=0,$ when the minimum is achieved at the boundary, and when the minimum is achieved at an interior point respectively.
Therefore, we have $\hat{G}\geq c_1(t)$ for some $c_1(t)$ depends on $\M_0,$ $T,$ and $t$. This completes the proof of Lemma \ref{lc1-lem7-upperbound of tF}.
\end{proof}

\subsection{$C^2$ boundary estimates for $u_r^*$-- solution of \eqref{ap.1}}
\label{c2bforur*}
First, we will need to construct a subsolution to \eqref{ap.1}. Consider
\be\label{subsolution to 3.1}
\lu^*=\hat{u}_0^*-[(1+\al)\td{t}]^{\frac{1}{1+\al}}\sqrt{1-|\xi|^2},
\ee
then $\lus_t=-[(1+\al)\td{t}]^{-\frac{\al}{1+\al}}\sqrt{1-|\xi|^2}.$ Moreover, we have
\be\label{c2bf.0}
\begin{aligned}
&\ka^*[w^*\gas_{ik}\lus_{kl}\gas_{lj}]>[(1+\al)\td{t}]^{\frac{1}{1+\al}}\ka^*\lt[w^*\gas_{ik}(-\sqrt{1-|\xi|^2})_{kl}\gas_{lj}\rt]\\
&=[(1+\al)\td{t}]^{\frac{1}{1+\al}}(1, \cdots, 1).
\end{aligned}
\ee
Here, we have used the convexity of $\hat{u}_0^*.$
Therefore, we get
\[\lus_t+F_*^{-\al}(w^*\ga^*_{ik}\lus_{kl}\ga^*_{lj})w^*<0,\] which implies
$\lus$ is a subsolution of \eqref{ap.1}.

Next, we will show $|D^2u^*|$ is bounded on $\partial B_r\times(0, T].$
For our convenience, we will the second covariant derivatives of $v$ instead, i.e., we will consider the following equation
\be\label{c2bf.1}
\left\{
\begin{aligned}
v_t&=-F_*^{-\al}(\Lambda_{ij})\,\,&\mbox{in $U_r\times(0, T]$},\\
v&=\frac{\hat{u}_0^*}{\sqrt{1-r^2}}-[(1+\al)\td{t}]^{\frac{1}{1+\al}}\,\,&\mbox{on $\partial U_r\times[0, T]$},\\
v(\cdot, 0)&=\frac{\hat{u}_0^*}{\sqrt{1-|\xi|^2}}-1:=v_0\,\,&\mbox{on $U_r\times\{0\}.$}
\end{aligned}
\right.
\ee

\begin{lemm}
\label{c2bf-lem1} Let $v$ be the solution of \eqref{c2bf.1}, then the second tangential derivatives on the boundary satisfy
$|\bar{\nabla}_{\al\beta}v|\leq C$ on $\partial U_r\times(0, T]$ for $\al, \beta<n.$
\end{lemm}
\begin{proof}
Let $\ubar{v}=\frac{\hat{u}_0^*}{\sqrt{1-|\xi|^2}}-[(1+\al)\td{t}]^{\frac{1}{1+\al}}$ and
$\tau_\al,$ $\tau_\beta$ be some tangential vector fields on $\partial U_r.$ Since $v-\ubar{v}\equiv 0$
on $\partial U_r$ we have
\[\bar{\nabla}_{\al\beta}(v-\ubar{v})=-\bar{\nabla}_{\tau_n}(v-\ubar{v})II(\tau_\al, \tau_\beta)\,\,\mbox{on $\partial U_r$}.\]
Here $\tau_n$ is the interior unit normal vector field to $\partial U_r$ and $II$ denotes the second fundamental form of $\partial U_r.$
Therefore Lemma \ref{c2bf-lem1} follows from Lemma \ref{c1forur*-lem3}
\end{proof}

Next, we will show that $|\bar{\nabla}_{\al n}v|$ is bounded.
In the following, we will denote
\[\mathfrak{L}\phi:=\phi_t-\td{F_v}^{-2}\td{F_v}^{ij}\bar{\nabla}_{ij}\phi+\phi\td{F}_v^{-2}\sum_i\td{F_v}^{ii}\]
for any smooth function $\phi.$ Here $\td{F}_v(\Lambda_{ij})=F_*^\al(\Lambda_{ij}),$
$\td{F}^{ij}_v=\frac{\partial\td{F}_v}{\partial\Lambda_{ij}},$ and $\Lambda_{ij}=\bar{\nabla}_{ij}v-v\delta_{ij}.$
\begin{lemm}
\label{c2bf-auxilary lemma}
Let $v$ be a solution of $\eqref{c2bf.1}$, $\ubar{v}$ be the subsoltion of \eqref{c2bf.1} which is defined in the proof of Lemma \ref{c2bf-lem1}, and   $h=(v-\ubar{v})+A\lt(\frac{1}{\sqrt{1-r^2}}-x_{n+1}\rt)$. Then for any $A_1>0,$ there exists $A>0$ such that
\[\mathfrak{L}h>\frac{A_1}{\td{F}_v^2}\sum\td{F}_v^{ii}.\]
\end{lemm}
\begin{proof}
By equation (7.5) in \cite{RWX}, it's straightforward to see that $\mathfrak{L}\lt(\frac{1}{\sqrt{1-r^2}}-x_{n+1}\rt)\geq \td{F_v}^{-2}\td{F_v}^{ii}.$ To prove Lemma \ref{c2bf-auxilary lemma}
we only need to show there exists $A_2>0$ such that
\[\mathfrak{L}(v-\ubar{v})>-\lt(\ubar{v}_t+\td{F}^{-1}(\ubar{\Lambda}_{ij})\rt)-\frac{A_2}{\td{F}_v^2}\sum\td{F}_v^{ii},\]
where $\ubar{\Lambda}_{ij}=\bar{\nabla}_{ij}\ubar{v}-\ubar{v}\delta_{ij}.$
This is equivalent to show
\[v_t-\td{F}^{-2}_v\td{F}^{ii}_v\Lambda_{ii}-\ubar{v}_t+\td{F}^{-2}_v\td{F}^{ii}_v\ubar{\Lambda}_{ii}
+\frac{A_2}{\td{F}^2_v}\sum\td{F}^{ii}_v\geq-\ubar{v}_t-\td{F}^{-1}(\ubar{\Lambda}_{ij}).\]
This implies
\be\label{c2bfadd1}
-\frac{\al+1}{\td{F}_v}+\td{F}^{-2}_v\td{F}^{ii}_v\ubar{\Lambda}_{ii}
+\frac{A_2}{\td{F}^2_v}\sum\td{F}_v^{ii}\geq-\td{F}^{-1}(\ubar{\Lambda}_{ij}).\ee
By condition \eqref{condition 2} we know $\td{F}^{\frac{1}{\al}}$ is concave, which gives
\[\frac{1}{\al}\td{F}_v^{\frac{1}{\al}-1}\td{F}_v^{ij}\ubar{\Lambda}_{ij}\geq\td{F}^{\frac{1}{\al}}(\ubar{\Lambda}_{ij}).\]
Moreover, by condition \eqref{condition 4} we get \[\frac{1}{\al}\td{F}_v^{\frac{1}{\al}-1}\sum\td{F}^{ii}_v\geq 1.\]
Therefore, \[\textbf{l.h.s of \eqref{c2bfadd1}}\geq-\frac{1+\al}{\td{F}_v}+\td{F}^{-2}_v
\lt(\frac{\al\td{F}^{\frac{1}{\al}}(\ubar{\Lambda}_{ij})}{\td{F}^{\frac{1}{\al}-1}_v}\rt)
+\frac{A_2}{\td{F}^2_v}\cdot\frac{\al}{\td{F}_v^{\frac{1}{\al}-1}}.\]
\textbf{Claim}: $\al\td{F}^{\frac{1}{\al}}(\ubar{\Lambda}_{ij})+A_2\al
+\frac{1}{\td{F}(\ubar{\Lambda}_{ij})}\td{F}_v^{1+\frac{1}{\al}}\geq(1+\al)\td{F}_v^{\frac{1}{\al}}.$\\
Proof of the claim: A straightforward calculation yields $C_1<\td{F}(\ubar{\Lambda}_{ij})<C_2(T),$ where $C_1$ only depends on the initial hypersurface
$\M_0$ while $C_2$ also deopends on the time $T.$ When $\td{F}_v>C_3(T):=(\al+1)C_2(T)$ we have
\[\frac{1}{\td{F}(\ubar{\Lambda}_{ij})}\td{F}_v^{1+\frac{1}{\al}}\geq(1+\al)\td{F}_v^{\frac{1}{\al}};\]
when $\td{F}_v\leq C_3(T)$ we can choose $A_2=A_2(C_3, \al)>0$ such that $A_2\al>(1+\al)C_3^{\frac{1}{\al}}.$
Thus the claim holds. This completes the proof of Lemma \ref{c2bf-auxilary lemma}.
\end{proof}

\begin{lemm}Let $v$ be a solution of $\eqref{c2bf.1}$ and suppose $\tau_n$ is the interior unit normal vector field of $\p U_r$. We have
\label{c2bf-lem3}
$|\bar{\nabla}_{\al n}v|\leq C$ on $\partial U_r\times(0, T].$
\end{lemm}
\begin{proof}
Consider any fixed point $p$ on $\partial B_r,$ we may assume it to be $(0, \cdots, 0, -r)$. Choose a new coordinates
$\{\td{\xi}_1, \cdots, \td{\xi}_n\}$ around $p$ such that $p$ is the origin  and $\td{\xi}_n$ axis is the interior normal to
$\partial B_r$ at $p.$
Denote $T:=\partial_\al+\frac{1}{r}\lt(\td{\xi}_\al\partial_n-\td{\xi}_n\partial_\al\rt),$
where $\partial_i:=\frac{\partial}{\partial\td{\xi}_i}$ for $1\leq i\leq n$ and $\al<n.$
Let $\phi(\cdot, t):= u_r^*(\cdot, t)-\lus(\cdot, t),$ where $u_r^*$ is the solution of \eqref{ap.1} and
$\lus$ is given in \eqref{subsolution to 3.1}. Then at $t=0$ we have
\be\label{c2bf.7}
\phi(\cdot, 0)=u_r^*(\cdot, 0)-\lus(\cdot, 0)=0,\\
\ee
and
\be\label{c2bf.8}
T\phi(\cdot, 0)=0.
\ee
Now let's denote $\Omega_\e:=B_r\cap B_\e(p),$ where $B_{\epsilon}(p)$ is a ball centered $p$ with radius $\epsilon$, then for any $t\in (0, T],$ on $\partial B_r\cap B_\e(p)$ we have
$\phi(\cdot, t)\equiv 0.$ This implies for any $t\in (0, T],$ $|T\phi (\cdot , t)|\leq C|\td{\xi}|^2$ on $\partial B_r\cap B_\e(p) $.
Consider $\Psi:=T\phi-C|\td{\xi}|^2=Tu_r^*-T\lus-C|\td{\xi}|^2.$ In the following we denote
\[\td{G}(w^*\gas_{ik}(u_r^*)_{kl}\gas_{lj})=-F_*^{-\al}(\ka^*[w^*\gas_{ik}(u_r^*)_{kl}\gas_{lj}]).\] Then
\[\frac{\partial u_r^*}{\partial t}-\td{G}(w^*\gas_{ik}(u_r^*)_{kl}\gas_{lj})w^*=0,\] and
$$T=\frac{1}{r}\left(\xi_{\al}\frac{\p}{\p\xi_n}-\xi_n\frac{\p}{\p\xi_{\alpha}}\right)$$ is an angular derivative vector. By Lemma 29 of \cite{RWX}
we get
\be\label{c2bf.9}
\frac{\partial (Tu_r^*)}{\partial t}-\lt(\td{G}^{ij}(w^*\gas_{ik}(Tu_r^*)_{kl}\gas_{lj})\rt)w^*=0,
\ee
where if we denote $a_{ij}=w^*\gas_{ik}(u_r^*)_{kl}\gas_{lj}$ then $\td{G}^{ij}=\frac{\p \td{G}}{\p a_{ij}}.$
Recall that $\lus=\hat{u}_0^*-[(1+\al)\td{t}]^{\frac{1}{1+\al}}\sqrt{1-|\xi|^2},$ then we derive
$T\lus=T\hat{u}_0^*.$
Applying Lemma 15 of \cite{WX20} we know
\be\label{star}
\bar{\nabla}_{ij}\lt(\frac{u}{w^*}\rt)-\frac{u}{w^*}\delta_{ij}=w^*\gas_{ik}u_{kl}\gas_{lj}.
\ee Combining with
\eqref{c2bf.9} we obtain
\be\label{c2bf.10}
(Tv)_t-\td{G}^{ij}\bar{\nabla}_{ij}(Tv)+(Tv)\sum\td{G}^{ii}=0,
\ee
where $Tv=T\lt(\frac{u_r^*}{w^*}\rt)=\frac{Tu_r^*}{w^*},$ and we also note that here $\td{G}^{ij}=\td{F}_v^{-2}\td{F}_v^{ij}.$
Moreover, it's easy to see that $|\mathfrak{L}T\ubar{v}|\leq C_1\sum\td{G}^{ii}.$
Now, since $|\td{\xi}|^2=\sum\limits_{\al<n}\xi_{\al}^2+(\xi_n+r)^2$ we have
\[\lt|\mathfrak{L}\frac{|\td{\xi}|^2}{w^*}\rt|\leq C_2\sum \td{G}^{ii}.\]
Therefore, choosing $C>0$ large the function $\td{\Psi}:=Tv-T\ubar{v}-C\frac{|\td{\xi}|^2}{w^*}$ satisfies
\[\td{\Psi}(\cdot, 0)\leq 0\,\,\mbox{in $P^{-1}(\Omega_{\e})\times\{0\},$}\]
\[\td{\Psi}(\cdot, t)\leq 0\,\,\mbox{on $\partial P^{-1}(\Omega_\e)\times(0, T],$}\]
and $\lt|\mathfrak{L}\td{\Psi}\rt|\leq C_3\sum\td{G}^{ii}.$ When $A>0$ very large we have
\[h\geq \td{\Psi}\,\,\mbox{in $P^{-1}(\Omega_\e)\times\{0\}$ and on $\p P^{-1}(\Omega_{\epsilon})\times(0,T]$},\]
\[\mathfrak{L}(\tilde{\Psi}-h)\leq 0\, \, \mbox{ in $\Omega_{\epsilon}\times(0,T]$}.\]
Therefore,
at any point $(p, t)\in \partial U_r\times(0, T]$ we get
$h_n>\td{\Psi}_n.$ This implies $\bar{\nabla}_{\al n}v(p, t)\leq C_4.$ Similarly, by considering
$T\ubar{v}-Tv-C\frac{|\td{\xi}|^2}{w^*}$ we obtain $\bar{\nabla}_{\al n}v\geq C_5.$
\end{proof}

\begin{lemm}Let $v$ be a solution of $\eqref{c2bf.1}$ and suppose $\tau_n$ is the interior unit normal vector field of $\p U_r$. We have \label{c2bf-lem4}
$|\bar{\nabla}_{nn}v|\leq C$ on $\partial U_r\times(0, T].$
\end{lemm}
\begin{proof}
For any fixed point $(p, t)\in\p U_r\times(0, T]$ we may assume $\lt(\bar\nabla_{\al\beta}v(p, t)\rt),$ $1\leq\al, \beta<n$ to be diagonal.
Then at this point
\be\label{ku1.1}
\begin{aligned}
\Lambda_{ij}&=\left[\ju{cccc}{v_{11}-v&0&\cdots&v_{1n}\\
0&v_{22}-v&\cdots&v_{2n}\\
\vdots&\vdots&\ddots&\vdots\\
v_{1n}&v_{2n}&\cdots&v_{nn}-v}\right].
\end{aligned}
\ee
By Lemma 1.2 in \cite{CNS3}, we know if $v_{nn}$ is very large, the eigenvalues $\lambda_1, \cdots, \lambda_n$ of $\Lambda_{ij}$ are
\begin{align*}
\lambda_\al&=v_{\al\al}-v+o(1)\\
\lambda_n&=v_{nn}-v+O(1).
\end{align*}
Since $F_*$ is bounded we know $v_{nn}$ is bounded from above. $v_{nn}$ is bounded from below comes from the strict convexity of the flow.
\end{proof}

\subsection{$C^2$ global estimates for $u_r^*$-- solution of \eqref{ap.1}}
\label{c2gforur*}
In this subsection, we will still use the hyperbolic model and study the
equation \eqref{c2bf.1}. We will estimate $|\bar{\nabla}^2v|$
on $\bar{U}_r\times [0, T].$ Keep in mind that a bound on $|\bar{\nabla}^2v|$
yields a bound on $|D^2u_r^*|.$

\begin{lemm}
\label{c2gf-lem1}
Let v be the solution of \eqref{c2bf.1}. Denote the eigenvalues of $(\bar{\nabla}_{ij}v -v\delta_{ij})$ by
$\lambda[\bar{\nabla}_{ij}v -v\delta_{ij}]=(\lambda_1,\cdots, \lambda_n).$
Then, $\lambda[\bar{\nabla}_{ij}v -v\delta_{ij}]$ is bounded from above.
\end{lemm}
\begin{proof}
Differentiating \eqref{c2bf.1} twice we get
\be\label{c2gf.2}
(v_t)_i=\td{F}^{-2}\td{F}^{kl}\Lambda_{kli},
\ee
and
\be\label{c2gf.3}
(v_t)_{ij}=\td{F}^{-2}\lt\{\td{F}^{kl}\Lambda_{klij}+\td{F}^{pq, rs}\Lambda_{pqi}\Lambda_{rsj}\rt\}
-2\td{F}^{-3}\lt(\td{F}^{kl}\Lambda_{kli}\rt)\lt(\td{F}^{kl}\Lambda_{klj}\rt).
\ee
Set $M=\max\limits_{(p, t)\in\bar{U}_r\times[0, T]}\max\limits_{|\xi|=1,\xi\in T_p\mathbb{H}^n}(\log\Lambda_{\xi\xi}+Nx_{n+1}),$
where $N$ is a constant to be determined later and $x_{n+1}$ is the coordinate function.
By the discussion in Subsection \ref{c2bforur*} we already know that $|\lambda|$ is bounded on $\p U_r\times[0, T]$. Therefore, in the following, we
may assume $M$  is achieved at an interior point $(p_0, t_0)$ for some direction $\xi_0.$ Choosing an orthonormal frame
$\{\tau_1,\cdots , \tau_n\}$ around $p_0$ such that $\tau_1(p_0)=\xi_0$ and $\Lambda_{ij}(p_0, t_0)=\lambda_i\delta_{ij}.$

Now, let's consider the test function $\phi=\log\Lambda_{11}+Nx_{n+1}.$
At its maximum point $(p_0, t_0)$, we have
\be\label{c2g.critic}
0=\frac{\Lambda_{11i}}{\Lambda_{11}}+N(x_{n+1})_i
\ee
and
\[0\geq\frac{\Lambda_{11ii}}{\Lambda_{11}}-\frac{\Lambda^2_{11i}}{\Lambda^2_{11}}+N(x_{n+1})_{ii}.\]
Therefore,
\be\label{c2gf.4}
\begin{aligned}
&\phi_t-\td{F}^{-2}\td{F}^{ii}\phi_{ii}\\
&=\frac{({\Lambda}_{11})_t}{\Lambda_{11}}-\td{F}^{-2}\td{F}^{ii}
\lt\{\frac{\Lambda_{11ii}}{\Lambda_{11}}-\frac{\Lambda^2_{11i}}{\Lambda^2_{11}}+N(x_{n+1})_{ii}\rt\}\\
&=\frac{(v_t)_{11}-v_t}{\Lambda_{11}}-\td{F}^{-2}\td{F}^{ii}
\lt\{\frac{\Lambda_{ii11}+\Lambda_{ii}-\Lambda_{11}}{\Lambda_{11}}-\frac{\Lambda^2_{11i}}{\Lambda^2_{11}}+Nx_{n+1}\delta_{ii}\rt\}\\
&=\frac{1}{\Lambda_{11}}\td{F}^{-2}\lt\{\td{F}^{pq, rs}\Lambda_{pq1}\Lambda_{rs1}-2\td{F}^{-1}(\nabla_1\td{F})^2\rt\}\\
&+\frac{1}{\td{F}\Lambda_{11}}-\td{F}^{-2}\td{F}^{ii}
\lt\{\frac{\Lambda_{ii}-\Lambda_{11}}{\Lambda_{11}}-\frac{\Lambda^2_{11i}}{\Lambda^2_{11}}+Nx_{n+1}\delta_{ii}\rt\}.
\end{aligned}
\ee
By our assumption \eqref{condition 2} we know that $\td{F}^{\frac{1}{\al}}$ is concave, which implies
\[\td{F}^{pp, qq}\Lambda_{pp1}\Lambda_{qq1}+\lt(\frac{1}{\al}-1\rt)\td{F}^{-1}(\nabla_1\td{F})^2\leq 0.\]
Since
\[\td{F}^{pq, rs}\Lambda_{pq1}\Lambda_{rs1}=\td{F}^{pp, qq}\Lambda_{pp1}\Lambda_{qq1}
+\sum\limits_{p\neq q}\frac{\td{F}^{pp}-\td{F}^{qq}}{\lambda_p-\lambda_q}\Lambda^2_{pq1},\]
we have
\be\label{c2gf.5}
\begin{aligned}
&\phi_t-\td{F}^{-2}\td{F}^{ii}\phi_{ii}\\
&\leq\frac{2}{\Lambda_{11}}\td{F}^{-2}\sum\limits_{p>1}\frac{\td{F}^{pp}-\td{F}^{11}}{\lambda_p-\lambda_1}\Lambda^2_{11p}
+\frac{1-\al}{\td{F}\Lambda_{11}}\\
&-(Nx_{n+1}-1)\td{F}^{-2}\sum\td{F}^{ii}+\td{F}^{-2}\td{F}^{ii}\frac{\Lambda^2_{11i}}{\Lambda^2_{11}}.
\end{aligned}
\ee
Now let
\[I:=\{j: \td{F}^{jj}\leq 4\td{F}^{11}\},\,\, J:=\{j: \td{F}^{jj}>4\td{F}^{11}\},\]
then combining with \eqref{c2g.critic}, equation \eqref{c2gf.5} becomes
\be\label{c2gf.6}
\begin{aligned}
&\phi_t-\td{F}^{-2}\td{F}^{ii}\phi_{ii}\\
&\leq\frac{1-\al}{\td{F}\Lambda_{11}}-(Nx_{n+1}-1)\td{F}^{-2}\sum\td{F}^{ii}+4\al\td{F}^{-2}|I|\frac{\td{F}N^2(x_{n+1})_i^2}{\Lambda_{11}}.
\end{aligned}
\ee
Notice that here we used $\td{F}^{11}\leq\frac{\al\td{F}}{\Lambda_{11}}.$
Moreover, since $F_*$ is bounded, concave and satisfies \eqref{condition 3}, we have $\sum\td{F}^{ii}\geq\al F_*^{\al-1}>c_0.$ Choosing $N=2$ we can see that if $\phi$ achieves its maximum at an interior point
$(p_0, t_0),$ then at this point $|\bar{\nabla}^2 v|$ is bounded from above. Otherwise, the maximum is achieved at
$\lt(\partial U_r\times(0, T]\rt)\bigcup\lt(U_r\times\{0\}\rt).$ Therefore, Lemma \ref{c2gf-lem1} is proved.
\end{proof}

\bigskip
\section{Local estimates}
\label{localest}
In this section we want to show that there exists a subsequence of $\{u^*_r\}$ that converges to the desired solution $u^*$ of  \eqref{ss.3}.

\subsection{Local $C^1$ estimates}
\label{lc1}
In this subsection we will prove the local $C^1$ estimates for $v_r$. We will study \eqref{ap.1} and consider the local $C^1$ estimates for
$u^*_r$ instead. For readers' convenience, we rewrite equation \eqref{ap.1} here:
\be\label{lc1.1}
\left\{
\begin{aligned}
(u^*_r)_t&=-F_*^{-\al}(w^*\gamma^*_{ik}u^*_{kl}\gamma^*_{lj})w^*\,\,&\mbox{in $B_r\times(0, T]$}\\
u^*_r(\cdot, t)&=\hat{u}^*_0-[(1+\al)\td{t}]^{\frac{1}{1+\al}}\sqrt{1-r^2}\,\,&\mbox{on $\p B_r\times[0, T]$}\\
u^*_r(\cdot, 0)&=\hat{u}^*_0-\sqrt{1-|\xi|^2}:=u^*_0\,\,&\mbox{on $B_r\times\{0\},$}
\end{aligned}
\right.
\ee
where $\td{t}=t+(1+\al)^{-1}.$ In the following, denote $\td{F}=F_*^{\al}$ and when there is no confusion, we will omit the superscript $r.$
\begin{lemm}
\label{lc1-lem1}
Let $u^*_r$ be the solution of \eqref{lc1.1} and suppose $\p$ is some angular derivative. Then we have $|\partial u^{*}_r|$ is bounded on $\bar{B}_r\times[0, T].$
\end{lemm}
\begin{proof}
We assume that $\partial=\xi_1\frac{\partial}{\partial\xi_2}-\xi_2\frac{\partial}{\partial\xi_1}.$
Since $u^*_t=-\td{F}^{-1}(\ga_{ik}^*u^*_{ij}\gas_{lj})(w^*)^{1-\al},$ by Lemma 29 of \cite{RWX} we get
\[(\partial u^*)_t-(w^*)^{1-\al}\td{F}^{-2}\td{F}^{kl}\gas_{ik}(\partial u^*)_{ij}\gas_{jl}=0.\]
Using the maximum principle we get that $\partial u$ achieves its maximum and minimum on the parabolic boundary.
Lemma \ref{lc1-lem1} follows directly from Lemma \ref{c1forur*-lem4}.
\end{proof}

\begin{lemm}
\label{lc1-lem2}
Let $u^*_r$ be the solution of \eqref{lc1.1} and suppose $\p$ is some angular derivative. Then we have $\partial^2 u^{*}_r$ is bounded from above on $\bar{B}_r\times[0, T].$
\end{lemm}
\begin{proof}
Let's denote $\td{G}=-\td{F}^{-1}=-F_*^{-\al},$ note that when $F_*$ is concave so is $\td{G}.$ A straightforward calculation yields,
\[(\partial^2u^*)_t=(w^*)^{1-\al}\td{G}^{pq, rs}(\partial \td{a}_{pq})(\partial \td{a}_{rs})+(w^*)^{1-\al}\td{G}^{kl}(\partial^2 \td{a}_{kl}),\]
where $\td{a}_{kl}=\gas_{ki}u^*_{ij}\gas_{jl}.$ By Lemma 30 of \cite{RWX} we get
\[(\partial^2u^*)_t-(w^*)^{1-\al}\td{G}^{kl}\lt(\gas_{ki}(\partial^2u^*)_{ij}\gas_{lj}\rt)\leq 0,\]
here the inequality is due to $\td{G}$ is concave.
Therefore, $\partial^2u^*$ achieves its maximum at the parabolic boundary and Lemma \ref{lc1-lem2} follows directly from Lemma \ref{c1forur*-lem5}.
\end{proof}

\begin{lemm}
\label{lc1-lem3}
Let $u^*_r$ be the solution of \eqref{lc1.1} and suppose $\p$ is some angular derivative. Then there is a positive constant $b$ such that
$$\frac{(r-|\xi|)\p^2u_r^*}{\sqrt{1-|\xi|^2}}>-b[(1+\al)\tilde{t}]^{\frac{1}{1+\al}}$$ on $\bar{B}_r\times[0, T].$ Here, $b$ only depends on $u_0^*.$
\end{lemm}
\begin{proof}
Without loss of generality we assume
\[\xi=(\xi_1, 0, \cdots, 0),\,\, |\xi|=\xi_1,\,\,\mbox{and $\partial=\xi_1\frac{\partial}{\partial\xi_2}-\xi_2\frac{\partial}{\partial\xi_1}.$}\]
Then at $\xi$ we have
\be\label{lc1.2}
\partial^2u^*=\xi_1^2u^*_{22}-\xi_1u^*_1.
\ee
Notice that by the convexity of $u^*$ we have $u^*_{22}>0.$ Now at the point $\xi$, applying Lemma \ref{c0fur*-lem} we know
\[u^*>-a_1\sqrt{1-|\xi|^2}[(1+\al)\tilde{t}]^{\frac{1}{1+\al}}-\sqrt{1-r^2}[(1+\al)\tilde{t}]^{\frac{1}{1+\al}}.\]
Applying the convexity of $u_r^*$ again we get
\[-a_1\sqrt{1-|\xi|^2}[(1+\al)\tilde{t}]^{\frac{1}{1+\al}}-\sqrt{1-r^2}[(1+\al)\tilde{t}]^{\frac{1}{1+\al}}+(r-|\xi|)u_1^*<-\sqrt{1-r^2}[(1+\al)\tilde{t}]^{\frac{1}{1+\al}}.\]
Therefore,
\[u_1^*<\frac{a_1\sqrt{1-|\xi|^2}}{r-|\xi|}[(1+\al)\tilde{t}]^{\frac{1}{1+\al}}.\]
Combining with \eqref{lc1.2}, we get
\[\partial^2u^*>-|\xi|\frac{a_1\sqrt{1-|\xi|^2}}{r-|\xi|}[(1+\al)\tilde{t}]^{\frac{1}{1+\al}}.\]
This completes the proof of Lemma \ref{lc1-lem3}.
\end{proof}

\begin{lemm}
\label{lc1-lem5-upperbarrier for ur*}
 Suppose $u^*_r$ is the solution of \eqref{lc1.1}. Let $\frac{1}{2}<\td{r}<r<1,$ $\mathbb{S}^{n-1}(\td{r})=\{\xi\in\R^n\mid\sum\xi_i^2=\td{r}^2\}.$ For any $\hat{\xi}\in\mathbb{S}^{n-1}(\td{r}),$
there exists a function
\be\label{lc1.5}
h=-a\sqrt{1-|\xi|^2}+b_1\xi_1+\cdots+b_n\xi_n+d
\ee
such that $h(\hat{\xi})=u^{*}_r(\hat{\xi}, t)$ and $h(\xi)>u^{*}_r(\xi, t)$ for any
$(\xi, t)\in\lt(\mathbb{S}^{n-1}(\td{r})\setminus\{\hat{\xi}\}\rt)\times (0, T].$
\end{lemm}
\begin{proof}
The proof of this Lemma is a small modification of the proof of Lemma \ref{c1forur*-lem2}. For readers' convenience, we include it here.

Since $\ka^*[w^*\gas_{ik}h_{kl}\gas_{lj}]=a\delta_{ij},$ we can choose $a>0$ such that $a^\al\leq\frac{1}{c_0},$ where $c_0>0$ is the same as the one in Lemma
\ref{lc1-lem4-tG lower bound}.
This gives
\[\td{F}(\ka^*[w^*\gas_{ik}h_{kl}\gas_{lj}])<\td{F}(\ka^*[w^*\gas_{ik}(u^*_r(\cdot, t))_{kl}\gas_{lj}]).\]
By rotating the coordinate we may assume $\hat{\xi}=(\td{r}, 0, \cdots, 0).$ We choose
\be\label{lc1.6}
b_k=\frac{\p u^{*}_r}{\p\xi_k}(\td{r}, 0, \cdots, 0, t),\,\,k=2, 3, \cdots, n
\ee
and choose $b_1$ such that
\be\label{lc1.7}
u^{*}_r(\td{r}, 0, \cdots, 0, t)=-a\sqrt{1-\td{r}^2}+b_1\td{r}+d.
\ee
To choose $d$ we consider an arbitrary great circle $c(s)$ on $\mathbb{S}^{n-1}(\td{r})$
passing through $\hat{\xi}.$ For example, the circle
\[\xi_1=\td{r}\cos s, \xi_2=\td{r}\sin s, \xi_3=\cdots=\xi_n=0,\,\,\mbox{where $-\pi\leq s\leq \pi$.}\]
Let
\be\label{lc1.8}
\begin{aligned}
&\hat{F}(s)=(h-u^{*}_r)|_{c(s)}\\
&=[u^{*}_r(\td{r}, 0, \cdots, 0, t)+a\sqrt{1-\td{r}^2}-d]\cos s+b_2\td{r}\sin s+d-u^{*}_r(s, t)-a\sqrt{1-\td{r}^2},
\end{aligned}
\ee
where $u^{*}_r(s, t)=u^{*}_r(\td{r}\cos s, \td{r}\sin s, 0, \cdots, 0, t).$
We have $\hat{F}(0)=0, $ $\frac{d\hat{F}}{ds}(0)=0,$ and
\[\frac{d^2\hat{F}(s)}{ds^2}=[d-u^{*}_r(\td{r}, 0, \cdots, 0, t)-a\sqrt{1-\td{r}^2}]\cos s-b_2\td{r}\sin s-\frac{d^2u^{*}_r}{ds^2}.\]
When $-\frac{\pi}{4}\leq s\leq\frac{\pi}{4}$ we have
\[\frac{d^2\hat{F}(s)}{ds^2}\geq[d-u^{*}_r(\td{r}, 0, \cdots, 0, t)-a\sqrt{1-\td{r}^2}]\frac{1}{\sqrt{2}}-c_1,\]
where $c_1>0$ is a constant determined by Lemma \ref{lc1-lem1} and \ref{lc1-lem2}.
When $s\in[-\pi, -\frac{\pi}{4})\cup(\frac{\pi}{4}, \pi]$ we have
\be\label{lc1.9}
\begin{aligned}
\hat{F}(s)&=d(1-\cos s)+[u^{*}_r(\td{r}, 0, \cdots, 0, t)+a\sqrt{1-\td{r}^2}]\cos s+b_2\td{r}\sin s-u^{*}_r(s, t)-a\sqrt{1-\td{r}^2}\\
&\geq\lt(1-\frac{\sqrt{2}}{2}\rt)d-2a\sqrt{1-\td{r}^2}-2\max\limits_{-\pi\leq s\leq\pi}|u^{*}_r(s, t)|-\lt|\frac{du^{*}_r(s, t)}{ds}\rt|\\
&\geq \lt(1-\frac{\sqrt{2}}{2}\rt)d-c_2.
\end{aligned}
\ee
Choosing $d>0$ large we are done.
\end{proof}

\begin{lemm}
\label{lc1-lem6-lowerbound for vr}
Suppose $v_r$ is the solution of \eqref{c2bf.1}. Let $\frac{1}{2}<\td{r}<r<1,$ $\mathbb{S}^{n-1}(\td{r})=\{\xi\in\R^n\mid\sum\xi_i^2=\td{r}^2\}.$ For any $(\hat{\xi}, t)\in\mathbb{S}^{n-1}(\td{r})\times(0, T]$ we have $\bar{\nabla}_iv_r(\hat{\xi}, t)>c_3, \,\,1\leq i\leq n,$ for some constant $c_3$ that is independent of $\td{r}$ and $r.$
\end{lemm}
\begin{proof}
Without loss of generality we consider the derivative of  $v_r$ at the point $\hat{\xi}=(\td{r}, 0, \cdots, 0).$ Without causing confusion, in the following we will drop the subscript $r$.
A straightforward calculation yields (for details see the proof of Lemma 15 in \cite{WX20})
\be\label{lc1.10}
\begin{aligned}
\bar{\nabla}_iv&=\gas_{ik}u^*_k+v\xi_i\\
&=\lt(\delta_{ik}-\frac{\xi_i\xi_k}{1+w^*}\rt)u^*_k+v\xi_i\\
&=u^*_i-\frac{\xi_i\xi_k}{1+w^*}u^*_k+v\xi_i.
\end{aligned}
\ee
Then at $\hat{\xi},$ we have
\[\bar{\nabla}_1v=\sqrt{1-\td{r}^2}u_1^*+v\td{r}\]
and $\bar{\nabla}_iv=u^*_i$ for $i\geq 2.$
By Lemma \ref{lc1-lem5-upperbarrier for ur*} we obtain
\[u^*_1>h^*_1=\frac{a\td{r}}{\sqrt{1-\td{r}^2}}+b_1.\]
Therefore,
\[\bar{\nabla}_1v>a\td{r}+b_1\sqrt{1-\td{r}^2}+v\td r.\]
Note that,  by the Subsection \ref{c0forur*} we get
\[-(a_1+1)[(1+\al)\td{t}]^{\frac{1}{1+\al}}\sqrt{1-|\xi|^2}<u^*_r<-a_0[(1+\al)\td{t}]^{\frac{1}{1+\al}}\sqrt{1-|\xi|^2}.\]
This implies
\[-(a_1+1)[(1+\al)\td{t}]^{\frac{1}{1+\al}}<v(\cdot, t)<-a_0[(1+\al)\td{t}]^{\frac{1}{1+\al}},\]
which in turn gives an uniform lower bound for $\bar{\nabla}_1v.$
 When $i=2, \cdots, n,$  by Lemma \ref{lc1-lem1}, we conclude that $\bar{\nabla}_iv=u^*_i$ is bounded at $(\hat{\xi}, t)$ for any $t\in[0, T]$ directly. Therefore, Lemma \ref{lc1-lem6-lowerbound for vr} is proved.
\end{proof}

Following the steps of Lemma \ref{lc1-lem5-upperbarrier for ur*} we can show
\begin{lemm}
\label{lc1-lem8-lowerbarrier for ur*}
Suppose $u^*_r$ is the solution of \eqref{lc1.1}.
Let $\frac{1}{2}<\td{r}<r<1,$ $\mathbb{S}^{n-1}(\td{r})=\{\xi\in\R^n\mid\sum\xi_i^2=\td{r}^2\}.$ For any $\hat{\xi}\in\mathbb{S}^{n-1}(\td{r}),$
there exists a function
\be\label{lc1.5}
\ubar{h}=-\ubar{a}\sqrt{1-|\xi|^2}+a_1\xi_1+\cdots+a_n\xi_n-a
\ee
such that $h(\hat{\xi})=u^{*}_r(\hat{\xi}, t)$ and $h(\xi)<u^{*}_r(\xi, t)$ for any
$(\xi, t)\in\lt(\mathbb{S}^{n-1}(\td{r})\setminus\{\hat{\xi}\}\rt)\times (0, T].$
Here $\ubar{a}$ is determined by $c_1(t)$ of Lemma \ref{lc1-lem7-upperbound of tF}, $a_1, \cdots, a_n$ are constants, and $a>0$ is chosen to be
$$a=10\lt[\frac{b\sqrt{1-\td{r}^2}}{r-\td{r}}[(1+\al)\tilde{t}]^{\frac{1}{1+\al}}+C(|u_0^*|)[(1+\al)\tilde{t}]^{\frac{1}{1+\al}}
+\ubar{a}\sqrt{1-\td{r}^2}\rt],$$
where $b$ is from Lemma \ref{lc1-lem3}.
\end{lemm}

\begin{lemm}
\label{lc1-lem9-upperbound for vr}
Suppose $v_r$ is the solution of \eqref{c2bf.1}.
Let $\frac{1}{2}<\td{r}<r<1,$ $\mathbb{S}^{n-1}(\td{r})=\{\xi\in\R^n\mid\sum\xi_i^2=\td{r}^2\}.$ For any
$(\hat{\xi}, t)\in\mathbb{S}^{n-1}(\td{r})\times (0, T],$
we have $\bar{\nabla}_iv_r<2\ubar{a}+2a\sqrt{1-\td{r}^2}+c, \,\,1\leq i\leq n,$ where $c$ is from Lemma \ref{lc1-lem1} and $\ubar{a}$, $a$ are from Lemma \ref{lc1-lem8-lowerbarrier for ur*}.
\end{lemm}
\begin{proof}
Without loss of generality we consider the derivative of  $v_r(\cdot, t)$ at the point $\hat{\xi}=(\td{r}, 0, \cdots, 0).$ Without causing confusion, in the following we will drop the subscript $r$. By Lemma \ref{lc1-lem8-lowerbarrier for ur*} we have
\[u_1^*<\ubar{h}_1=\frac{\ubar{a}\td{r}}{\sqrt{1-\td{r}^2}}+a_1\]
Therefore, at $\hat{\xi}$ we get
\[\bar{\nabla}_1v<\ubar{a}\td{r}+a_1\sqrt{1-\td{r}^2}+v\td{r}.\]
Since $-\ubar{a}\sqrt{1-\td{r}^2}+a_1\td{r}-a=u^*(\hat{\xi}, t),$ we obtain
\[a_1\td{r}\sqrt{1-\td{r}^2}=u^*(\hat{\xi}, t)\sqrt{1-\td{r}^2}+a\sqrt{1-\td{r}^2}+\ubar{a}(1-\td{r}^2),\]
which yields the upper bound for $\bar{\nabla}_1 v$ at $\hat{\xi}$. Exactly the same as the proof of Lemma \ref{lc1-lem6-lowerbound for vr}, we can show
$\bar{\nabla}_iv$ is bounded from above for $i\geq 2.$ Thus, Lemma \ref{lc1-lem9-upperbound for vr} is proved.
\end{proof}

\subsection{Local $C^2$ estimates}
\label{lc2}
First, let's note that for any $\xi_0\in B_1,$ there exists a linear function $l(\xi)$ such that
$l(\xi)-u_0^*(\xi)>\delta_1>0$ in a small neighborhood $U_{\xi_0}\subset B_1$ of $\xi_0,$ where $\delta_1>0$ is a small constant.
Moreover, $u_0^*(\xi)-l(\xi)>0$ on $\partial B_r$ for $r>r_0.$ From equation \eqref{lc1.1} and Lemma \ref{lc1-lem4-tG lower bound} we know
$\lt(u^*_r\rt)_t< 0,$  this yields
\[l(\xi_0)-u^*_r(\xi_0, t)> l(\xi_0)-u_0^*(\xi_0)\,\,\mbox{for $t>0$}.\]
Moreover, on $\p B_r,$ when $t\leq T$ and $r>r_0(T)$ we have $$u^*_r(\xi, t)-l(\xi)=u_0^*(\xi)+\sqrt{1-r^2}-[(1+\al)\td{t}]^{\frac{1}{1+\al}}\sqrt{1-r^2}-l(\xi)>0.$$
Therefore, we can use $l(\xi)$ to construct cutoff function and obtain the following local $C^2$ estimates.

\begin{lemm}
\label{lc2-lem1}
Let $u^*_r$ be the solution of \eqref{lc1.1}
For any $\xi_0\in B_1$ and $T>0,$ we have $$\lt|\ka^*[w^*\gas_{ik}(u^*_r)_{kl}\gas_{lj}](\xi_0, t)\rt|<C$$ for $t\in[0, T]$ and $1>r>r_0(T).$
Here $C=C(T)$ is independent of $r.$
\end{lemm}
\begin{proof}
From the discussion in Section 2 of \cite{WX20} we know that $\lambda(\bar{\nabla}_{ij}v-v\delta_{ij})=\ka^*[w^*\gas_{ik}(u^*_r)_{kl}\gas_{lj}]$
for $v=\frac{u^*_r}{w^*}.$
Therefore, in this proof we will consider equation \eqref{c2bf.1} and obtain an upper bound for $|\lambda|$ instead.
Let $l(\xi)$ be a linear function described as above and let $\td{s}=\frac{l(\xi)}{w^*},$ then it's easy to see that
\[\bar{\nabla}_{ij}\td{s}=\td{s}\delta_{ij}.\]
In the following, we denote
\[\eta=\td{s}-v\,\,\mbox{and $\vartheta=v^2-|\bar{\nabla}v|^2.$}\]
We also denote $\td{G}=-\td{F}^{-1},$ where $\td{F}=F_*^\al.$
Following \cite{SUW}, let's consider $$W(p, t)=\eta^\beta\exp\Phi(\vartheta)\Lambda_{\xi\xi},$$ where $\xi\in T_p\mathbb{H}^{n}$ is a unit vector. we may assume $W$ achives its maximum at $(p_0, t_0)$ for some direction $\xi_0$. Choosing an orthonormal frame $\{\tau_1, \cdots, \tau_n\}$ around
$p_0$ such that $\tau_1(p_0)=\xi_0$ and $\Lambda_{ij}(p_0, t_0)=\lambda_i\delta_{ij}.$ Differentiating $W$ at $(p_0, t_0)$ we obtain
\be\label{lc2.5}
0=\frac{W_i}{W}=\frac{\beta\eta_i}{\eta}+\Phi'\vartheta_i+\frac{\Lambda_{11i}}{\Lambda_{11}},
\ee
and
\be\label{lc2.6}
\begin{aligned}
0&\geq\frac{\beta\eta_{ii}}{\eta}-\frac{\beta\eta_i^2}{\eta^2}+\Phi'\vartheta_{ii}+\Phi''(\vartheta_i)^2\\
&+\frac{\Lambda_{ii11}}{\Lambda_{11}}+\frac{\Lambda_{ii}}{\Lambda_{11}}-1-\lt(\frac{\Lambda_{11i}}{\Lambda_{11}}\rt)^2,
\end{aligned}
\ee
where we have used the equality $\Lambda_{11ii}=\Lambda_{ii11}+\Lambda_{ii}-\Lambda_{11}.$
Moreover, at $(p_0, t_0)$ we have
\be\label{lc2.7}
\begin{aligned}
0&\leq\frac{\beta}{\eta}\mathcal{L}\eta+\Phi'\mathcal{L}\vartheta+\frac{1}{\Lambda_{11}}\mathcal{L}\Lambda_{11}\\
&+\beta\td{G}^{ii}\lt(\frac{\eta_i}{\eta}\rt)^2-\Phi''\td{G}^{ii}(\vartheta_i)^2+\td{G}^{ii}\lt(\frac{\Lambda_{11i}}{\Lambda_{11}}\rt)^2.
\end{aligned}
\ee
It's easy to obtain that
\be\label{lc2.8}
\mathcal{L}\eta=-\eta\td{G}^{ii}-(1+\al)\td{G},
\ee
and
\be\label{lc2.9}
\mathcal{L}\Lambda_{11}=\td{G}^{pq, rs}\Lambda_{pq1}\Lambda_{rs1}+(\al-1)\td{G}+\Lambda_{11}\sum\td{G}^{ii}.
\ee
Moreover, a straightforward calculation yields
\be\label{lc2.1}
\begin{aligned}
(\vartheta)_t&=2vv_t-2v_k(v_k)_t\\
&=2v\td{G}-2v_k\td{G}^{ii}\Lambda_{iik},
\end{aligned}
\ee
\be\label{lc2.2}
\vartheta_i=2vv_i-2v_kv_{ki}=2v_i\Lambda_{ii},
\ee
and
\be\label{lc2.3}
\begin{aligned}
\vartheta_{ii}&=2v_i^2+2vv_{ii}-2v^2_{ki}-2v_kv_{kii}\\
&=2v_i^2+2vv_{ii}-2v^2_{ii}-2v_k(\Lambda_{iik}+v_i\delta_{ik})\\
&=-2\Lambda^2_{ii}-2v\Lambda_{ii}-2v_k\Lambda_{iik}.
\end{aligned}
\ee
Thus we get
\be\label{lc2.4}
(\vartheta)_t-\td{G}^{ii}\vartheta_{ii}=2v\td{G}(1-\al)+2\td{G}^{ii}\Lambda^2_{ii}.
\ee
Therefore, \eqref{lc2.7} becomes
\be\label{lc2.10}
\begin{aligned}
0&\leq\frac{\beta}{\eta}(-\eta\sum\td{G}^{ii}-(1+\al)\td{G})\\
&+\Phi'\lt(2v\td{G}(1-\al)+2\td{G}^{ii}\Lambda^2_{ii}\rt)\\
&+\frac{1}{\Lambda_{11}}\lt[\td{G}^{pq, rs}\Lambda_{pq1}\Lambda_{rs1}+(\al-1)\td{G}+\Lambda_{11}\sum\td{G}^{ii}\rt]\\
&+\beta\td{G}^{ii}\lt(\frac{\eta_i}{\eta}\rt)^2-\Phi''\td{G}^{ii}(\vartheta_i)^2+\td{G}^{ii}\lt(\frac{\Lambda_{11i}}{\Lambda_{11}}\rt)^2.
\end{aligned}
\ee
Denote $I=\{j: \td{g}_j\leq 4\td{g}_1\}$ and $J=\{j: \td{g}_j>4\td{g}_1\}.$
When $j\in I$ we have,
\be\label{lc2.11}
\begin{aligned}
&\td{g}_j\lt(\frac{\Lambda_{11j}}{\Lambda_{11}}\rt)^2=\td{g}_j\lt(\frac{\beta\eta_j}{\eta}+\Phi'\vartheta_j\rt)^2\\
&\leq(1+\e)(\Phi')^2\td{g}_j(\vartheta_j)^2+\lt(1+\frac{1}{\e}\rt)\beta^2\td{g}_j\lt(\frac{\eta_j}{\eta}\rt)^2.
\end{aligned}
\ee
When $j\in J$ we have,
\be\label{lc2.12}
\begin{aligned}
&\beta\td{g}_j\lt(\frac{\eta_j}{\eta}\rt)^2=\frac{1}{\beta}\td{g}_j\lt(\Phi'\vartheta_j+\frac{\Lambda_{11j}}{\Lambda_{11}}\rt)^2\\
&\leq\frac{1+\e}{\beta}(\Phi')^2\td{g}_j(\vartheta_j)^2+\frac{(1+\e^{-1})}{\beta}\td{g}_j\lt(\frac{\Lambda_{11j}}{\Lambda_{11}}\rt)^2.
\end{aligned}
\ee
Therefore,
\be\label{lc2.13}
\begin{aligned}
&\beta\sum\limits_{j=1}^n\td{g}_j\lt(\frac{\eta_j}{\eta}\rt)^2+\sum\limits_{j=1}^n\td{g}_j\lt(\frac{\Lambda_{11j}}{\Lambda_{11}}\rt)^2\\
&\leq 4n[\beta+(1+\e^{-1})\beta^2]\td{g}_1\frac{|\bar{\nabla}\eta|^2}{\eta^2}\\
&+(1+\e)(1+\beta^{-1})(\Phi')^2\sum\limits_{j=1}^n\td{g}_j(\vartheta_j)^2\\
&+[1+(1+\e^{-1})\beta^{-1}]\sum\limits_{j\in J}\td{g}_j\lt(\frac{\Lambda_{11j}}{\Lambda_{11}}\rt)^2.
\end{aligned}
\ee

Combining \eqref{lc2.13} with \eqref{lc2.10} we get
\be\label{lc2.14}
\begin{aligned}
0&\leq (-\beta+1)\sum\td{G}^{ii}-\frac{(1+\al)\td{G}\beta}{\eta}+2v\td{G}(1-\al)\Phi'+\frac{(\al-1)\td{G}}{\Lambda_{11}}\\
&+2\Phi'\sum\td{g}_i\Lambda_{ii}^2+\frac{1}{\Lambda_{11}}\td{G}^{pq, rs}\Lambda_{pq1}\Lambda_{rs1}\\
&+4n[\beta+(1+\e^{-1})\beta^2]\td{g}_1\frac{|\nabla\eta|^2}{\eta^2}\\
&+[(1+\e)(1+\beta^{-1})(\Phi')^2-\Phi'']\sum\limits_{j=1}^n\td{g}_j(\vartheta_j)^2\\
&+[1+(1+\e^{-1})\beta^{-1}]\sum\limits_{j\in J}\td{g}_j\lt(\frac{\Lambda_{11j}}{\Lambda_{11}}\rt)^2.
\end{aligned}
\ee
Here we point out that $\bar{\nabla}_i\td{s}=\gas_{ik}\frac{\p l(\xi)}{\p\xi_k}+\td{s}\xi_i,$ which yields $|\bar{\nabla}\eta|$ is bounded.
Let $\Phi(\vartheta)=-\log(\vartheta+A),$ where $A=A(T, \M_0)>0$ large such that $\vartheta+A>1.$
Note that $\td{g}_j\vartheta^2_j=4\td{g}_jv^2_j\Lambda_{jj}^2,$ we want
\[[(1+\e)(1+\beta^{-1})(\Phi')^2-\Phi'']4|\nabla v|^2+\Phi'<0.\] This is  equivalent to
\be\label{lc2.15}
(\e+\beta^{-1}+\e\beta^{-1})4|\nabla v|^2-(\vartheta+A)<0.
\ee
By choosing $\e>0$ small and $\beta>0$ large inequality \eqref{lc2.15} can be satisfied. Then \eqref{lc2.14}
can be reduced to
\be\label{lc2.16}
\begin{aligned}
0&\leq\frac{C_1}{\eta}+C_2-\frac{1}{\vartheta+A}\sum\td{g}_i\Lambda_{ii}^2\\
&+\frac{1}{\Lambda_{11}}\td{G}^{pq, rs}\Lambda_{pq1}\Lambda_{rs1}+\frac{C_3(\beta+\beta^2)\td{g}_1}{\eta^2}\\
&+[1+(1+\e^{-1})\beta^{-1}]\sum\limits_{j\in J}\td{g}_j\lt(\frac{\Lambda_{11j}}{\Lambda_{11}}\rt)^2.
\end{aligned}
\ee
Since $\td{g}$ is concave we have
\[\frac{1}{\Lambda_{11}}\td{G}^{pq, rs}\Lambda_{pq1}\Lambda_{rs1}\leq\frac{2}{\Lambda_{11}}\sum\limits_{j\in J}\frac{\td{g}_1-\td{g}_j}{\Lambda_{11}-\Lambda_{jj}}\Lambda_{11j}^2\leq-\frac{3}{2\Lambda_{11}^2}\sum\limits_{j\in J}\td{g}_j\Lambda^2_{11j}.\]
When $1+\frac{1}{\beta}+\frac{1}{\beta\e}<\frac{3}{2}$ we have
\be\label{lc2.17}
0\leq \frac{C_1}{\eta}+C_2-\frac{1}{\vartheta+A}\sum\td{g}_i\Lambda^2_{ii}+\frac{C_3(\beta+\beta^2)\td{g}_1}{\eta^2}.
\ee
Since $\td{g}_i=\al F_*^{-\al-1}F_*^{ii},$ we get
$\td{g}_1<\frac{\al|\td{G}|}{\lambda_1.}$ By \eqref{lc2.17} and \eqref{condition 6} we conclude
\[\frac{1}{\vartheta+A}|\td{G}|\lambda_1\leq\frac{C_1}{\eta}+C_2+\frac{C_3(\beta+\beta^2)|\td{G}|}{\lambda_1\eta^2}.\]
Therefore, $\lambda_1\eta\leq C_4$ and the Lemma is proved.
\end{proof}

In Section \ref{approximate problem} we proved the existence of the solution $u^*_r$ to equation \eqref{ap.1}. In Section \ref{localest}
we obtained the local estimates of $\{u^*_r\}_{r>r_0}$ in $K\times[0, T]$ for any compact subset $K\subset B_1$ and $T>0$ Therefore, we can find a subsequence of
$u^*_r$ such that $u^*_r\goto u^*,$ and $u^*$ satisfies \eqref{ss.3} in $[0, T].$ Since $T>0$ is arbitrary, we obtain the existence of a solution to equation \eqref{ss.3} in $[0, \infty)$. Applying the standard maximum principle, we also know that the solution to \eqref{ss.3} is unique.  The Legendre transform of $u^*,$ which we denote by $u,$ is the desired entire solution of $\eqref{ss.1}.$

\bigskip
\section{Convergence}
\label{conv}
In this section, we will show after rescaling, the entire solution to equation \eqref{ss.1}
converges to the unit future hyperboloid as $t\rightarrow \infty$.

Let $\td{X}=\frac{X}{[(1+\al)\td{t}]^{\frac{1}{1+\al}}}$ and $\td{\tau}=\int_0^t[(1+\al)\td{t}]^{-1}dt.$ Then $\td{X}$ satisfies
\be\label{conv.1}
\begin{aligned}
(\td{X})_{\td{\tau}}&=(\td{X})_t\cdot\frac{\p t}{\p \td{\tau}}\\
&=F^{\al}(\ka)[(1+\al)\td{t}]^{\frac{\al}{1+\al}}\nu-\frac{X}{[(1+\al)\td{t}]^{\frac{1}{1+\al}}}.
\end{aligned}
\ee
Since $\td{\ka}=[(1+\al)\td{t}]^{\frac{1}{1+\al}}\ka,$ \eqref{conv.1} becomes
\be\label{conv.2}
(\td{X})_{\td{\tau}}=F^\al(\td{\ka})\nu-\td{X}.
\ee
Denote $\td{v}=\langle\td{X}, \nu\rangle,$ then under this rescaling we get
\be\label{conv.3}
\td{v}_{\td{\tau}}=-\lt(F_*^{-\al}(\td{\Lambda}_{ij})+\td{v}\rt),
\ee
where $\td{\Lambda}_{ij}=\bar{\nabla}_{ij}\td{v}-\td{v}\delta_{ij}.$
A straightforward calculation gives
\be\label{conv.4}
\begin{aligned}
(\td{\Lambda}_{ii})_{\td{\tau}}&=\lt(\bar{\nabla}_{ii}\td{v}\rt)_{\td{\tau}}-\td{v}_{\td{\tau}}\\
&=-\bar{\nabla}_{ii}F_*^{-\al}-\bar{\nabla}_{ii}\td{v}+F_*^{-\al}(\td{\Lambda})+\td{v}\\
&=-\bar{\nabla}_{ii}F_*^{-\al}+F_*^{-\al}(\td{\Lambda})-\td{\Lambda}_{ii}.
\end{aligned}
\ee
This yields
\be\label{conv.5}
\begin{aligned}
\lt(F_*^{-\al}\rt)_{\td{\tau}}&=-\al F_*^{-\al-1}F_*^{kk}(\td{\Lambda}_{kk})_{\td{\tau}}\\
&=\al F_*^{-\al-1}F_*^{kk}\bar{\nabla}_{kk}F_*^{-\al}-\al F_*^{-2\al-1}\sum F_*^{kk}+\al F_*^{-\al}.
\end{aligned}
\ee
In the following, we denote $L:=\frac{\p}{\p\td{\tau}}-\al F_*^{-\al-1}F_*^{kk}\bar{\nabla}_{kk}$
and $\Phi=F_*^{-\al}(\td{\Lambda})+\td{v}.$ Then we have
\[L\Phi=-\lt(1+\al F_*^{-\al-1}\sum F_*^{kk}\rt)\Phi.\]
Consider $\Psi=e^{2\td{\tau}}\Phi^2,$ then $\Psi$ satisfies
\be\label{conv.6}
L\Psi\leq-2\al F_*^{-\al-1}\sum F_*^{kk}\Psi\leq 0.
\ee
Therefore, by the maximum principle we know $\Phi^2\leq e^{-2\td{\tau}}c_0.$ This implies $\Phi\goto 0$ as $\td{\tau}\goto\infty.$
Since $\td{v}=\frac{\lt<X, \nu\rt>}{[(1+\al)\td{t}]^{\frac{1}{1+\al}}}=\frac{v}{[(1+\al)\td{t}]^{\frac{1}{1+\al}}},$ by Lemma \ref{c0fur*-lem}
we have $-(a_1+1)\leq\td{v}\leq-a_0.$ From earlier estimates for $F_*^{-\al}(\td{\Lambda})+\td{v},$ we get when $\td{\tau}>\tau_0>0,$
\be\label{convadd.1}
c_1\leq F_*^{-\al}(\td{\Lambda})\leq c_2.
\ee
Now suppose $\td{u}^*_\infty=-\sqrt{1-|\xi|^2},$ it's easy to see that $\td{u}^*_\infty$ satisfies
\be\label{conv.7}
\left\{\begin{aligned}
F_*^{-\al}(w^*\gas_{ik}(u^*_{\infty})_{kl}\gas_{lj})w^*&=-u^*_\infty\,\,&\mbox{in $B_1$}\\
u^*_\infty&=0\,\,&\mbox{on $\p B_1$}.
\end{aligned}\right.
\ee
Moreover, the graph of the Legendre transform of $\td{u}^*_\infty,$ denoted by $\td{\M}_\infty=\{(x, \td{u}_\infty(x))\mid x\in\R^n\},$ is a hyperboloid.

Let $u^*$ be the solution of \eqref{ss.3} and $\td{u}^*$ be the rescaling of $u^*.$ Then $\td{u}^*$ satisfies
\be\label{conv.7'}
\left\{\begin{aligned}
\lt(\td{u}^*\rt)_{\td{\tau}}&=-F_*^{-\al}(w^*\gas_{ik}(\td{u}^*)_{kl}\gas_{lj})w^*-\td{u}^*\,\,&\mbox{in $B_1\times(0, \infty)$}\\
\td{u}^*(\cdot, \td{\tau})&=0\,\,&\mbox{on $\p B_1\times(0, \infty)$}\\
\td{u}^*(\cdot, 0)&=u^*_0\,\,&\mbox{on $B_1\times\{0\}.$}
\end{aligned}\right.
\ee
\begin{lemm}
\label{convlem1}
Let $\td{u}^*$ be the solution of \eqref{conv.7'}. Then for
each sequence $\td{\tau}_j\goto\infty$ there is a subsequence $\td{\tau}_{j_k}\goto\infty$ such that $\td{u}^*(\cdot, \td{\tau}_{j_k})\goto\td{u}^*_{\infty}(\cdot)$ locally smoothly on $K,$ where $K\subset B_1$ is a compact set.
\end{lemm}
\begin{proof} Without loss of generality we assume \eqref{convadd.1} holds for $\td{\tau}\geq0.$ In view of \eqref{convadd.1} we obtain
\be\label{convadd.2}
c_1[(1+\al)\td{t}]^{-\frac{\al}{1+\al}}\leq F_*^{-\al}(\Lambda)\leq c_2[(1+\al)\td{t}]^{-\frac{\al}{1+\al}}.
\ee
In Lemma \ref{lc1-lem5-upperbarrier for ur*} we can choose $a=c_3[(1+\al)\td{t}]^{\frac{1}{1+\al}},$ where $0<c_3<\lt(\frac{1}{c_2}\rt)^{\frac{1}{\al}},$
then we obtain
\be\label{convadd.3}
\bar{\nabla}_iv>c_4[(1+\al)\td{t}]^{\frac{1}{1+\al}},
\ee
for any $1\leq i\leq n$ and $\xi\in\mathbb{S}^{n-1}(\td{r}),$ $\frac{1}{2}<\td{r}<1.$
Similarly by choosing $\ubar{a}=c_5[(1+\al)\td{t}]^{\frac{1}{1+\al}}$ where $c_5>\lt(\frac{1}{c_1}\rt)^{\frac{1}{\al}}$
in Lemma \ref{lc1-lem8-lowerbarrier for ur*}, we get
\be\label{convadd.4}
\bar{\nabla}_iv<c_6[(1+\al)\td{t}]^{\frac{1}{1+\al}}.
\ee
Now in Lemma \ref{lc2-lem1} we will choose $\td{s}=-\frac{c}{w^*}[(1+\al)\td{t}]^{\frac{1}{1+\al}},$ where $c>0$ is a small number.
Moreover, we let $\Phi(\vartheta)=-\log(\vartheta+A[(1+\al)\td{t}]^{\frac{2}{1+\al}}),$ here $A>0$ such that
$\vartheta+A[(1+\al)\td{t}]^{\frac{2}{1+\al}}>[(1+\al)\td{t}]^{\frac{2}{1+\al}}.$ Then by Lemma \ref{lc2-lem1} we get in
$K\times[0, \infty)$ we have
\[\Lambda_{\max}\leq c_7[(1+\al)\td{t}]^{\frac{1}{1+\al}}.\]
This combining with \eqref{convadd.1} and condition \eqref{condition 7} yields in $K\times[0,\infty)$
\[0< c_8\leq \td{\Lambda}_{\min}\leq \td{\Lambda}_{\max}\leq c_7.\]
Applying standard regularity and convergence theorems, we prove Lemma \ref{convlem1}.
\end{proof}

\begin{rmk}
The solutions of the following equation
$$F^{\alpha}(\kappa)=-\langle X,\nu\rangle$$
is called the self-expender.  According to Husiken \cite{Hui}, the singularity of the mean curvature flow of a strictly closed convex hypersurface is a self-shrinker in Euclidean space.  Here, since the unit normal is time-like, we have the opposite sign.
\end{rmk}

\begin{rmk}
We want to point out that \textbf{Conditions A} is invariant under Lorentz transformation. It's clear that the conditions (1),(2),(3)
of \textbf{Conditions A} are invariant. We only need to check condition (4). Let $u_0$ satisfy
$$\sqrt{|x|^2+C_0}< u_0(x)<\sqrt{|x|^2+C_1},$$ for some suitable positive constant $C_0,C_1$.
Recall that the Lorentz transformation  is
\be\label{cb1.2}
\left\{
\begin{aligned}
x_1'&=\frac{x_1-\alpha x_{n+1}}{\sqrt{1-\alpha^2}}\\
x_i'&=x_i\\
x'_{n+1}&=\frac{x_{n+1}-\alpha x_1}{\sqrt{1-\alpha^2}}.
\end{aligned}
\right.
\ee
It's easy to check that
$$(x'_{n+1})^2<(x'_1)^2+\cdots+(x'_n)^2+C_1$$
provided that $(x_{n+1})^2<(x_1)^2+\cdots+(x_n)^2+C_1,$
and  $$(x'_{n+1})^2>(x'_1)^2+\cdots+(x'_n)^2+C_0,$$
provided that $(x_{n+1})^2>(x_1)^2+\cdots+(x_n)^2+C_0$.
\end{rmk}

\end{document}